\documentclass[MSNbibl,number,citesort,seceqn,dvips]{arxbj}

\aid{0}
\volume{20}
\issue{1}
\pubyear{2014}
\firstpage{334}
\lastpage{376}
\doi{10.3150/12-BEJ489} 

\makeatletter
\def\mid{\vert}
\newcommand{\rrvert}{\vert}
\newcommand{\llvert}{\vert}
\newcommand{\eqref}[1]{(\ref{#1})}

\newcommand{\esp}{\mathbb{E}}
\newcommand{\var}{\operatorname{var}}
\newcommand{\cov}{\operatorname{cov}}

\newtheorem{theo}{Theorem}[section]
\newtheorem{lem}[theo]{Lemma}
\newtheorem{prop}[theo]{Proposition}
\newtheorem{cor}[theo]{Corollary}

\newremark{rem}{Remark}[section]
\newremark{example}{Example}[section]

\makeatother

\begin{document}
\begin{frontmatter}

\title{Estimating the scaling function of multifractal measures and
multifractal random walks using ratios}

\runtitle{Estimating the scaling function}

\begin{aug}
\author[1]{\fnms{Carenne} \snm{Lude\~{n}a}\thanksref{1}\ead[label=e1]{carinludena@gmail.com}}%
\and
\author[2]{\fnms{Philippe} \snm{Soulier}\corref{}\thanksref{2}\ead[label=e2]{philippe.soulier@u-paris10.fr}}
\runauthor{C. Lude\~{n}a and P. Soulier} 
\address[1]{Departamento de Matem\'aticas, Universidad central de Venezuela,
Ciudad Universitaria, Los \mbox{Chaguaramos}, Caracas, Venezuela. \printead{e1}}
\address[2]{Laboratoire MODAL'X, D{\'e}partement de Math{\'e}matiques, UFR
SEGMI, Universit{\'e} Paris Ouest-Nanterre, 200 avenue de la R\'epublique,
92001 Nanterre Cedex, France.\\ \printead{e2}}
\end{aug}

\received{\smonth{2} \syear{2011}}
\revised{\smonth{3} \syear{2012}}

\begin{abstract}
In this paper, we prove central limit theorems for bias reduced
estimators of
the structure function of several multifractal processes, namely mutiplicative
cascades, multifractal random measures, multifractal random walk and
multifractal fractional random walk as defined by Lude\~{n}a [\textit{Ann. Appl. Probab.}
\textbf{18} (2008) 1138--1163]. Previous
estimators of the structure functions considered in the literature were
severely biased with a logarithmic rate of convergence, whereas the estimators
considered here have a polynomial rate of convergence.
\end{abstract}

%

\end{frontmatter}

\section{Introduction}

A random process $X=\{X(s), s\in[0,T]\}$ ($T>0$) with stationary
increments will
be called multifractal if its scaling behaviour is characterized by a strictly
concave function $\zeta$, called the scaling function, such that for a certain
range of real numbers $q$
\[
\esp \bigl[ \bigl|X(t) - X(s) \bigr|^q \bigr] = c(q) |t-s|^{\zeta(q)} .
\]
If the function $\zeta$ is linear, then the process is said to be
monofractal, as is the case, for instance, for the fractional Brownian
motion (FBM) $B_H$, $0<H<1$, which is defined as a continuous centered
Gaussian process such that $B_H(0)=0$ and for all $s,t \geq0$,
\[
\operatorname{var} \bigl(B_H(t)-B_H(s) \bigr) =
|t-s|^{2H} .
\]
Then, for all $q>-1$, $\esp[|B_H(t) - B_H(s)|^q] = c(q) |t-s|^{qH}$,
with $c(q) = \esp[|B_H(1)|^q]$.

Several truly multifractal processes with stationary increments have been
defined. The earliest one is the multiplicative cascade introduced by
Mandelbrot \cite{mandelbrot1974} and rigorously studied by
Kahane and
Peyri{\`e}re \cite{kahanepeyriere1976}. These processes were generalized by
Barral and
Mandelbrot \cite{barralmandelbrot2002}, Muzy and Bacry \cite{muzybacry2002} and
Bacry and Muzy \cite{bacrymuzy2003}. The latter authors introduced multifractal random
measures (MRM) and multifractal random walks (MRW) as time changed Brownian
motion. Lude{\~n}a \cite{ludena2008} and Abry \textit{et~al.} \cite{abrychainaiscoutinpipiras2009}
introduced multifractal (fractional) random walks which are conditionally
fractional Gaussian processes.

For these processes, multifractality results from a distributional scaling
property which can be written as
\[
\bigl\{ X(\lambda t) , 0 \leq t \leq T \bigr\} \stackrel{\mathrm{law}}= \bigl
\{U_\lambda X(t) , 0 \leq t \leq T \bigr\}
\]
for $0<\lambda<1$, $U_\lambda$ is a positive random variable
independent of the
process $X$ such that $\esp[U_\lambda^q] = \lambda^{\zeta(q)}$ for
$q<q_{\max}$
a certain parameter depending on the process under consideration (and with
certain additional restrictions on the values of $\lambda$ for which this
identiy holds in the case of multifractal cascades, see
Section~\ref{secmultcasc}). For the models, we will formally introduce
in the
sequel, it is defined as
\[
q_{\max} = \sup \bigl\{q\dvt \zeta(q) \geq1 \bigr\} .
\]
It is also important to note that the fixed time horizon $T$ beyond
which this scaling property need not be true is finite, except for
monofractal processes such as the FBM.

Given a multifractal process observed discretely on $[0,T]$, it is of
obvious interest to be able to identify the scaling function $\zeta$.

Let $t_1,\ldots, t_N$, with $t_i-t_{i-1}=\Delta=T/N$ be a regular
partition of
$[0,T]$ (typically on a dyadic scale). Typically, for $q<q_{\max}$,
$\zeta(q)$
is estimated by calculating logarithms of the empirical structure function
\[
S_N(X,q) := \sum_{j=0}^{N-1} |
\Delta X_{j}|^q ,
\]
where $\Delta X_{j} = X((j+1)\Delta)- X(j\Delta)$. Estimators of
$\zeta$ can then be defined by
\begin{eqnarray*}
\hat{\zeta}_N(q) & :=& 1 + \frac{\log_2(S_{N}(X,q))}{\log_2(\Delta)} ,
\\
\tilde{\zeta}_N(q) & :=& 1 + \log_2 \biggl(
\frac{S_{N}(X,q)}{
S_{2N}(X,q)} \biggr) .
\end{eqnarray*}
These estimators have been thoroughly dealt with for multiplicative
cascades in
Ossiander and
Waymire \cite{ossianderwaymire2000}. The authors show that $\hat{\zeta}_N(q)$ and
$\tilde{\zeta}_N(q)$ are consistent estimators of $\zeta(q)$ for
$q<q_0$, where
$q_0<q_{\max}$ is the largest value of $q$ such that
\[
\zeta(q) - q \zeta'(q) < 1 .
\]
For $q>q_0$, $\hat{\zeta}_N(q)$ is seen to converge almost surely to a linear
function of $q$. Moreover, conditional central limit theorems (where the
limiting distribution is a mixture of normal laws) are seen to hold for suitably
normalized versions of both estimators if $2q<q_0$. However, as shown
in~Ossiander and
Waymire \cite{ossianderwaymire2000}, the convergence rates for these
estimators are
very different. The rate of convergence of $\hat\zeta_N(q)$ is of order
$\log_2(N)$ because of the existence of a bias term, whereas we will
show that of $\tilde\zeta_N(q)$ is a power of $N$ which depends on~$\zeta$.

In order to enlarge the domain of consistency of the estimators and obtain
unconditional central limit theorems, the so-called mixed asymptotic framework
has been introduced by allowing the number $L$ of basic observations intervals
to increase with $N$. In the case of multiplicative cascades and MRM, the
processes over different intervals are independent. The observations are
$X((jL+k)\Delta)$, $0 \leq j \leq L-1$, $0 \leq k \leq N-1$, and the estimators
are now modified as follows
\begin{eqnarray*}
\hat{\zeta}_{L,N}(X,q) & :=& 1 + \frac{\log_2(S_{L,N}(X,q))}{\log_2(\Delta)} ,
\\
\tilde{\zeta}_{L,N}(X,q) & :=& 1 + \log_2 \biggl(
\frac{S_{L,N}(X,q)}{
S_{L,2N}(X,q)} \biggr) ,
\end{eqnarray*}
with
\[
S_{N,L}(X,q) := \sum_{j=0}^{L-1}
\sum_{k=0}^{N-1} |\Delta X_{jL+k}|^q
.
\]
The mixed asymptotic framework for multiplicative cascades has been recently
developed in Bacry \textit{et~al.} \cite{bacrygloterhoffmannmuzy2010}. The authors show that if
$L=[N^{ \chi}]$, where $[x]$ stands for the greatest integer $m\le x$ with
$\chi>0$, then $\hat{\zeta}_{N,L}(X,q)$ is consistent for $q<q_\chi$ where
$q_\chi$ is the largest value of $q$ such that
\[
\zeta(q) - q\zeta'(q) < \chi+1 .
\]
Note that as $\chi$ tends to infinity, $q_\chi$ might become greater than
$q_{\max}$, so we will only consider values of $\chi$ such that
$q_\chi<q_{\max}$.

However, once again, there exists a bias term $b_N:=\mathbb
{E}[M_1^q]/\log_2(N)$, which
entails slow convergence of the estimator.
In analogy to the nonmixed asymptotic framework it is reasonable to consider
ratio based estimators such as $\tilde{\zeta}_{N,L}(X,q)$ in order to improve
convergence rates. It turns out, as follows quite straightforwardly
from the
results of Bacry \textit{et~al.} \cite{bacrygloterhoffmannmuzy2010}, that
$\tilde{\zeta}_{N,L}(X,q) \to\zeta(q)$, a.s., for a dyadic partition,
but the
authors failed to prove a central limit theorem, although they hint at
it at the
end of their Section~3. Almost sure convergence for dyadic partitions,
or in
probability for general partitions, of $\hat{\zeta}_{N,L}(X,q)$ has
also been
recently considered by Duvernet \cite{duvernet2009} for $\chi\ge0$ and $X$ a Brownian
MRW or a MRM. However, the author does not prove central limit theorems nor
establish convergence rates in either case. An interesting application for
testing whether a process is a semimartigale or a multifractal process is
developed in Duvernet, Robert and
Rosenbaum \cite{duvernetrobertrosenbaum2010} which is based on the
limiting behaviour of variation ratios, but the authors restrict their attention
to log-normal multifractal random walks and $q=2$.

The main goal of this paper is to obtain central limit theorems for the
estimator $\tilde\zeta_{N,L}$ in the mixed asymptotic setting, for
multiplicative cascades, multifractal random measures (MRM) and multifractal
random walks (MRW) that are either a time changed Brownian motion or a more
general process related to a fractional Brownian motion with Hurst index
$H>1/2$. Our main results in all these cases state unconditional
central limit
theorems with polynomial rates of convergence, contrary to $\hat\zeta_{L,N}$
which can only achieve logarithmic rates of convergence, and to the
case $L=1$
where only conditional central limit theorems can be obtained.

For multiplicative cascades, Ossiander and
Waymire \cite{ossianderwaymire2000} also considered
negative values of~$q$ such that $\esp[M^q([0,1])]<\infty$ and $0>q>\inf_{h\le
0}\{h \psi'(h)-\psi(h)<1\}$. However, we cannot extend such a result
in full
generality in the present context, since for certain MRM which are considered
here, $\esp[M^q([0,1])]=\infty$ for all $q<0$. Moreover, negative
moments of the
Gaussian law are infinite for $q\leq-1$, thus even if the MRM
considered has
finite negative moments, that might not be the case for the MRW. For these
reasons, and not to increase the length of the paper, we do not
consider the
case $q<0$.

The rest of the paper is organized as follows. We will consider multiplicative
cascades in Section~\ref{secmultcasc}, MRM in Section~\ref{secmrm}, and
MRW in
Section~\ref{secmrw}. Section~\ref{seclemmata} contains the main ideas
of the
proofs and technical lemmas are relegated to the \hyperref[app]{Appendix}. To the best
of our
knowledge, our results are the first to deal with the MRW in the case $H>1/2$.

\section{Multiplicative cascades}
\label{secmultcasc}
In this section, we give a precise formulation of consistency results for
$\tilde{\zeta}(q)$, whenever $q<{q}_\chi$, and a central limit theorem whenever
$2q<{q}_\chi$, in the case of multiplicative cascades. The results are a
straightforward application of previous results of
Bacry \textit{et~al.} \cite{bacrygloterhoffmannmuzy2010} and
Ossiander and
Waymire \cite{ossianderwaymire2000}. However, they provide the framework for dealing
with both MRM and MRW so will be dealt with in some detail. Before we
state the
main results, we shall introduce the mixed asymptotic setting, following
Bacry \textit{et~al.} \cite{bacrygloterhoffmannmuzy2010}.

For any given $n$-tuple $r$ and $i<n$ set $r|i=(r_1,\ldots, r_i)$ and
if $s$ is
an $i$-tuple and $v$ an $(n-i)$-tuple set $r=s*v$ to be the resulting $n$-tuple
obtained by concatenation. For each $j\in\mathbb{Z}$ and fixed $T$, set
$I^{(j)}:=[jT,(j+1)T]$. Over each $I^{(j)}$ we will construct an independent
{\emph multiplicative cascade} as defined in Mandelbrot \cite{mandelbrot1974}. For this,
consider a collection $\{W_r^{(j)}, r\in\{0,1\}^n , n\ge1, j \in
\mathbb
Z\}$ of independent random variables with common law $W$ such that $\esp[W]=1$
and $\esp[W\log_2W]<1$ and for each $n\geq1$ and $j\in\mathbb Z$,
consider the
random measure defined by
\[
\lambda_n^{(j)}(I) = T2^{-n} \sum
_{\{r\in\{0,1\}^n : (j-1+r)T \in
I^{(j)}\}} \prod_{i=1}^n
W_{r|i}^{(j)}
\]
for any Borel subset $I$ of $I^{(j)}$, and each $r=(r_1,\ldots,r_n) \in
\{0,1\}^n$ is associated to the real number $\sum_{i=1}^n r_i2^{n-k}$.
It can be seen (see Kahane and
Peyri{\`e}re \cite{kahanepeyriere1976},
Ossiander and
Waymire \cite{ossianderwaymire2000} for details on the construction and main
results) that there exists a random measure $\lambda_\infty^{(j)}$,
such that
\[
\mathbb P \bigl(\lambda_n^{(j)} \Rightarrow
\lambda_\infty^{(j)} \mbox{ as } n \to\infty \bigr) = 1 ,
\]
where $\Rightarrow$ stands for vague convergence. The limiting measure verifies
$\esp[\lambda_\infty^{(j)}([0,T])] = T$. By construction $\lambda_\infty^{(j)}$
are independent random measures, defined over the disjoint intervals
$I^{(j)}$. Set $\lambda_\infty:= \sum_{j\in\mathbb Z} \lambda_\infty^{(j)}$.

Set $\mathcal{F}_n = \sigma\{W_{r}^{(j)}, r\in\{0,1\}^n, j\in\mathbb
Z\}$ and let $\Delta_{k,n}^{(j)}:=[(j +k2^{-n})T,(j+(k+1)2^{-n})T]$,
$k=0,\ldots, 2^{n}-1$, be the $k$th diadic interval at level $n$, of
the interval $I^{(j)}$. Then,
\[
\lambda_\infty\bigl(\Delta_{k,n}^{(j)}\bigr) =
2^{-n} Z_{j,k,n} \prod_{i=1}^n
W_{r_n(k) \mid i}^{(j)} ,
\]
where for each $n$, $Z_{j,k,n}$, $0 \leq k <2^n$, $j\in\mathbb Z$, are
i.i.d. random variables with the same distribution as
$\lambda_\infty([0,T])$ and independent of $\mathcal F_n$, and
$r_n(k)$ is the dyadic representation of $k$, that is, $k=\sum_{i=1}^n
r_{n,i}(k) 2^{n-i}$ for $k<2^n$. Moreover, $Z_{j,2k,n+1}$ and
$Z_{j,2k+1,n+1}$ are independent of $Z_{j,k',n}$ for $k' \ne k$.
The above identity straightforwardly yields the
scaling property:
\[
\esp \bigl[ \lambda_\infty^q\bigl(\Delta_{k,n}^{(j)}
\bigr) \bigr] = 2^{-n\zeta(q)} \esp \bigl[\lambda_\infty^q
\bigl([0,T] \bigr) \bigr] ,
\]
with
\[
\zeta(q) = q - \log_2 \bigl(\esp\bigl[W^q\bigr] \bigr) .
\]
It is shown in Kahane and
Peyri{\`e}re \cite{kahanepeyriere1976} that for $q>1$, the condition
$\zeta(q)>1$ implies $\esp[\lambda_\infty^q([0,T])] < \infty$.

\begin{example}
Consider the log-normal cascade, where $\log W = \mu+\sigma Z$ and $Z$
is a
standard Gaussian random variable. The condition $\esp[W]=1$ implies
that $\mu
= -\sigma^2/2$. Then it is easily obtained that
\begin{eqnarray*}
\zeta(q) &=& q - \frac{q(q-1)\sigma^2}{2\log2} ,\qquad q_{\max} = \biggl(\frac{2\log2}{\sigma^2}
\biggr)\vee1 , \qquad q_{0} = \frac{\sqrt{2\log
2}}{\sigma} ,\\
 q_\chi&=&
\frac{\sqrt{2(1+\chi)\log2}}\sigma .
\end{eqnarray*}

\end{example}

Denote
\[
S_{L,n}(q) = \sum_{j=0}^{L-1}
\sum_{k=0}^{2^n-1} \lambda_\infty^q
\bigl(\Delta_{k,n}^{(j)}\bigr)
\]
and
\begin{eqnarray*}
\hat{\zeta}(q) := 1 - \frac{\log_2(S_{L,n}(q))}{n} , \qquad\tilde{\zeta}(q) = 1 +
\log_2 \biggl(\frac{S_{L,n}(q)}{S_{L,n+1}(q)} \biggr) .
\end{eqnarray*}
Note that although in the asymptotics $L$ will eventually depend on
$n$, its
value is the same in the quantities $S_{L,n}$ and $S_{L+1,n}$.

\textit{Consistency}.
For each $n\geq1$, let $\{\xi, \xi_{j,k,n}, 0\leq j \leq L-1, 0 \leq k
\leq
2^n-1\}$ be a collection of i.i.d. random variables, independent of
$\mathcal
F_n$. Define
\[
\tilde S_{n,q} = 2^{-nq} \sum_{j=0}^{L-1}
\sum_{k=0}^{2^n-1} \prod
_{i=1}^n \bigl( W_{r_n(k)\mid i}^{(j)}
\bigr)^{q} \xi_{j,k,n} .
\]
In Bacry \textit{et~al.} \cite{bacrygloterhoffmannmuzy2010}, the following general result is
shown to hold.
%
\begin{prop}
\label{propconsistency-cascade}
For $\chi>0$, assume that $L=[2^{n\chi}]$, $q < {q}_\chi$ and there exists
$\varepsilon>0$ such that $\esp[\xi^{1+\varepsilon}]<\infty$. If $\xi$ is
nonnegative, then
\[
L^{-1} 2^{-n} 2^{n\zeta(q)} \bigl(\tilde
S_{n,q}-\esp[\tilde S_{n,q}]\bigr) \to0\qquad \mbox{a.s.}
\]
\end{prop}

Note that by construction $\esp[\tilde S_{n,q}] = L 2^n 2^{-n\zeta(q)}
\esp[\xi]$, so that the above result yields the almost sure convergence $L^{-1}
2^{-n} 2^{n\zeta(q)} \tilde S_{n,q} \to\esp[\xi]$ under the stated conditions.
As a consequence, by the definition of $S_{L,n}(q)$,
Proposition~\ref{propconsistency-cascade} yields
%
\begin{equation}
\label{eqstructfun} L^{-1} 2^{-n} 2^{n\zeta(q)}
S_{L,n}(q) \to\esp\bigl[\lambda_\infty^q
\bigl([0,T]\bigr)\bigr] \qquad\mbox{a.s. }
\end{equation}
for $q<{q}_\chi$. Then, clearly,
\[
\hat{\zeta}(q) - \zeta(q) + \chi+ \frac{\log_2\esp[\lambda_\infty^q([0,T])]}{n} \to0 \qquad\mbox{a.s.} ,
\]
and (\ref{eqstructfun}) also implies that $\tilde{\zeta}(q) \to\zeta(q)$
a.s.
On the other hand, if $q > q_\chi$, then Bacry \textit{et~al.}~\cite{bacrygloterhoffmannmuzy2010}
show that $\hat{\zeta}(q)\to\zeta'(q_\chi)q,$ which is a linear
function of
$q$. In this case, $\tilde{\zeta}(q)$ is also not consistent as the normalized
structure function tends to zero~(Ossiander and
Waymire \cite{ossianderwaymire2000}).

\textit{Central limit theorem}.
Based on Proposition~\ref{propconsistency-cascade}, it is also possible to
obtain a central limit theorem for $\tilde{\zeta}(q)$. We remark that
in the
mixed asymptotic framework the limiting variance is deterministic. The
proof of
the central limit theorem follows from a series of corollaries of the following
general result for the mixed framework which is a direct generalization of
Proposition~4.1 in Ossiander and
Waymire \cite{ossianderwaymire2000} and
Proposition~\ref{propconsistency-cascade}. We first state some general
notation. Let $\{\xi,\xi_{j,k,n},0\leq j \leq L-1,0\leq k \leq2^{n-1},
n\geq0\}$ be as above and define
\begin{eqnarray*}
V_{n,q} = 2^{-2nq} \sum_{j=0}^{L-1}
\sum_{k=0}^{2^n-1} \prod
_{i=1}^n\bigl(W_{r_n(k) \mid i}^{(j)}
\bigr)^{2q} ,\qquad  R_{n,q} = \tilde S_{n,q}/V_{n,q}^{1/2}
.
\end{eqnarray*}
The following proposition is seen to hold true as a direct
generalization of Proposition 4.1 in Ossiander and
Waymire \cite{ossianderwaymire2000}, whenever
$2q<{q}_\chi$.
%
\begin{prop}
\label{propOW-41}
If $2q<{q}_\chi$, $\esp[\xi_{j,k,n}] = 0$, $\esp[\xi_{j,k,n]}^2] =
\sigma^2$ and if
\[
\sup_n \sup_{j,k} \esp \bigl[ |\xi_{j,k,n}|^{2(1+\delta)}
\bigr] < \infty
\]
for some $\delta>0$, then
\[
\lim_{n\to\infty} \esp \bigl[ \mathrm{e}^{\mathrm{i}zR_{n,q}} \mid
\mathcal{F}_n \bigr] = \mathrm{e}^{-\sigma^2z^2/2}
\]
and $R_{n,q}$ converges weakly to the centered Gaussian law with
variance $\sigma^2$.
\end{prop}
The proof follows exactly as that of Proposition 4.1 in
Ossiander and
Waymire \cite{ossianderwaymire2000}, using
Proposition~\ref{propconsistency-cascade}. The latter also yields that $L^{-1}
2^{-n} 2^{n\zeta(2q)} V_{n,q}$ converges to 1 a.s. We can now state a central
limit theorem for the empirical structure function.
%
\begin{prop}
\label{propclt-cascade}
If $2q<{q}_\chi$, then
\[
L^{-1/2} 2^{-n/2} 2^{n\zeta(2q)/2} \bigl\{ S_{L,n}(q)
-2^{\zeta(q)-1} S_{L,n+1}(q) \bigr\} \to_d \mathbf{N}
\bigl(0,V(q)\bigr) ,
\]
with
\[
V(q) = \var \bigl( Z_0^q - 2^{\zeta(q)-1-q} \bigl
\{Z_1^q W_1^q +
Z_2^q W_2^q \bigr\} \bigr)
\]
and $Z_1$, $Z_2$ are i.i.d. with the same
distribution as $\lambda_\infty([0,1])$ and independent of $W_1$,
$W_2$, which
are i.i.d. with the same distribution as $W$, and $Z_0 =
(Z_1W_1+Z_2W_2)/2$ has
the same distribution as $\lambda_\infty([0,1])$.
\end{prop}
\begin{pf}
The proof follows from Proposition \ref{propOW-41}, by noting that $
S_{L,n}(q) - 2^{\zeta(q)-1} S_{L,n+1}(q)$ can be expressed as
\begin{eqnarray*}
S_{L,n}(q) - 2^{\zeta(q)-1} S_{L,n+1}(q) = 2^{-nq}
\sum_{j=0}^{L-1} \sum
_{k=0}^{2^n-1} \prod_{i=1}^n
\bigl(W_{r_n(k)\mid i}^{(j)}\bigr)^{q} \xi_{j,k,n}
\end{eqnarray*}
with
\begin{eqnarray*}
\xi_{j,k,n} = Z_{j,k,n}^q - 2^{\zeta(q)-1-q} \bigl\{
Z_{j,2k,n+1}^q W_{r_n(k)*0}^q +
Z_{j,2k,n+1}^q W_{r_n(k)*1}^q \bigr\} ,
\end{eqnarray*}
since $r_n(k)*0 = r_{n+1}(2k)$ and $r_n(k)*1 = r_{n+1}(2k+1)$.
Indeed, the random variables $\xi_{j,k,n}$, $j\in\mathbb Z$, $0\leq k
<2^n$, are
i.i.d. (for each fixed $n$) and it clearly holds that $\esp[\xi_{j,k,n}] =0$,
$\esp[\xi_{j,k,n}^{2}] = V(q)$ and $\esp[|\xi_{j,k,n}|^{2+\delta
}]<\infty$,
whenever $2q<q_{max}$ for small enough $\delta>0$.
\end{pf}
Thus we obtain a central limit theorem for $\tilde\zeta(q)$.
%
\begin{theo}
\label{theocor46}
Assume $2q<{q}_\chi$. Then
\[
2^{n(1 + \chi+ 2\psi(q) - \psi(2q))/2} \bigl\{ \tilde{\zeta}(q) - \zeta(q) \bigr\} \to_d
\mathbf{N}\bigl(0,V(q)/\bigl(\esp\bigl[\lambda_\infty^q
\bigl([0,T]\bigr)\bigr]\bigr)^2\bigr).
\]
\end{theo}
\begin{pf}
By Proposition~\ref{propconsistency-cascade} and~(\ref{eqstructfun}),
$S_{L,n+1}(q)2^{ \zeta(q)-1} /S_{n,L}(q) \to1$ a.s. so
\begin{eqnarray*}
\tilde{\zeta}(q) - \zeta(q) & =& \log_2 \biggl( \frac{S_{L,n}(q)} {
2^{\zeta(q)-1} S_{L,n+1}(q)}
\biggr) = - \log_2 \biggl(1 - \frac{S_{L,n}(q)-2^{\zeta(q)-1}S_{L,n+1}(q)} {S_{L,n}(q)} \biggr)
\\
& =& \frac{S_{L,n}(q)-2^{\zeta(q)-1}S_{L,n+1}(q)} {S_{L,n}(q)} \times \bigl\{1+\mathrm{o}_P(1) \bigr\} .
\end{eqnarray*}
The proof is concluded by applying Proposition~\ref{propclt-cascade} and
noting that $2^{-n\chi}L\to1$.
\end{pf}

\section{Multifractal random measures}
\label{secmrm}
Once again we are interested in the mixed asymptotic framework defined
by the
parameter~$\chi$. The main ideas dealt with in this section are very
similar in
spirit to those in Duvernet \cite{duvernet2009}. We include the proofs for completeness'
sake, since they are very similar to those which will be developed to study
multifractal random walks. We recall the main definition and properties of
multifractal random measures, hereafter MRM, following Bacry and Muzy \cite{bacrymuzy2003}.
Start by defining for $l>0$, $w_l(u)=P(A_l(u))$ and set
\[
M(I) = \lim_{l\to0} \int_I \mathrm{e}^{w_l(u)}
\,\mathrm{d} u ,
\]
where $I$ is any Borel set in $\mathbb{R}$. Here $P$ is an independently
scattered random measure on $\mathcal{S}^+=\{(s,t), t>0\}$ such that
$P(\bigcup_{i=1}^\infty A_i) = \sum_{i=1}^\infty P(A_i)$ if the Borel measurable
sets $A_i$ are pairwise disjoint and then the random variables $P(A_i)$,
$i\geq1$, are independent, and
%
\begin{equation}
\label{eqlaplace} \esp\bigl[\mathrm{e}^{qP(A)}\bigr] = \mathrm{e}^{\psi(q)\mu(A)} ,
\end{equation}
with $\mu(A)=\int_A t^{-2} \,\mathrm{d} s \,\mathrm{d} t$ and
\[
A_l(u) = \bigl\{(s,t), u-(t/2 \wedge T/2) < s < u + (t/2 \wedge
T/2), t>l\bigr\} .
\]

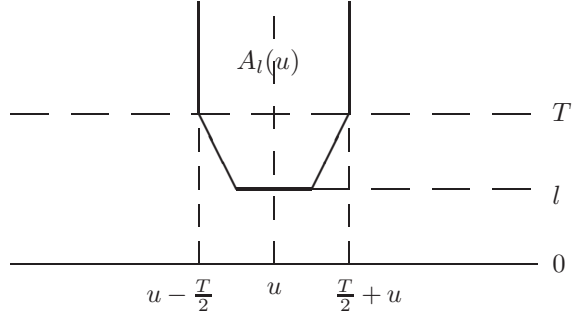
\begin{figure}[t]\vspace*{10pt}
\centering
\setlength{\unitlength}{1mm}
\begin{picture}(100,40)(-50,-28)

\multiput(0,-10)(8,0){4}{\line(1,0){5}}
\put(-40,-20){\line(1,0){70}}
\multiput(-40,0)(8,0){9}{\line(1,0){5}}

\multiput(-15,-20)(0,5){5}{\line(0,1){3}}
\multiput(5,-20)(0,5){5}{\line(0,1){3}}
\multiput(-5,-20)(0,6){6}{\line(0,1){3}}

\put(32,-12){$l$}
\put(32,-1){$T$}
\put(32,-21){$0$}
\put(-22,-25){$u-\frac T2$}
\put(3,-25){$\frac T2+u$}
\put(-10,6){$A_l(u)$}
\put(-6,-24){$u$}

\thicklines
\put(-10,-10){{\line(1,0){10}}}
\put(-15,0){{\line(0,1){15}}}
\put(-10,-10){{\line(-1,2){5}}}
\put(5,0){{\line(0,1){15}}}
\put(0,-10){{\line(1,2){5}}}

\end{picture}
\caption{The set $A_l(u)$.}
\label{figAlu}
\end{figure}
It is readily checked that $\mu(A_l(t)) = T + \log(T/l)$, which implies,
with~(\ref{eqlaplace}), that
%
\begin{equation}
\label{eqlaplace-wl} \esp\bigl[ \mathrm{e}^{qw_l(t)}\bigr] =
\mathrm{e}^{(T+\log T)\psi(q)} l^{-\psi
(q)} .
\end{equation}
The function $\psi$ is the log-Laplace transform of the infinitely divisible
random measure $P$, assumed to exist for $q<q^*$, for some $q^*>1$. It
is convex
and satisfies $\psi(0)=\psi(1)=0$. By the L\'evy--Khinchine
representation theorem,
it can be expressed as
\[
\psi(q) = \frac{\sigma^2}2 + mq + \int_{-\infty}^\infty
\bigl\{\mathrm e^{qx} - 1 - x \mathbf1_{\{|x|\leq1\}} \bigr\} \nu(
\mathrm{d} x) ,
\]
where $\nu$ is the L\'evy measure of $P$ and satisfies
\[
\int_{-\infty}^\infty\bigl(x^2\wedge1\bigr)
\nu(\mathrm{d} x) < \infty .
\]
The assumption that $\psi(q)$ is finite for $q<q^*$ entails the
following condition. For all $q<q^*$,
\[
\int_{1}^\infty\mathrm{e}^{qx} \nu(
\mathrm{d} x) < \infty .
\]
By Bacry and Muzy \cite{bacrymuzy2003}, Theorem~4, there exists a certain infinitely
divisible random variable~$\Omega_\lambda$, which is independent of $M([0,T])$,
such that $\esp[\mathrm{e}^{q\Omega_\lambda}] = \lambda^{-\psi(q)}$ and for
$\lambda,l\in(0,1)$,
%
\begin{equation}
\label{eqscaling} \bigl\{w_{\lambda l}(\lambda u) , 0 \leq u \leq T\bigr\}
\stackrel{\mathrm{law}}= \bigl\{w_{l}(u) + \Omega_\lambda , 0 \leq u
\leq T\bigr\} .
\end{equation}
The latter is known as the scaling property. This implies that
%
\begin{equation}
\label{eqcdv-mrm} M\bigl([0,\lambda T]\bigr) \stackrel{d} {=} \lambda
\mathrm{e}^{\Omega_\lambda}M\bigl([0,T]\bigr)
\end{equation}
for $\lambda\in[0,1]$, so that
%
\begin{equation}
\label{eqscaling-mrm} \esp\bigl[M^q\bigl([0,\lambda T]\bigr)\bigr] =
\lambda^{\zeta(q)} m(q)
\end{equation}
with $\zeta(q) = q-\psi(q)$ and $m(q)=\esp[M^q([0,T])]$. It is shown in
Bacry and Muzy \cite{bacrymuzy2003}, Theorem~3, that if $\zeta(q)>1$, then
$\mathbb{E}[M^q([0,T])]<\infty$. As previously, set $q_{\max}$ to be the
greatest value of $q$ such that $\zeta(q) \ge1$ and for $\chi\geq0$, define
$q_\chi$ as
\[
q_\chi= \max\bigl\{q\dvt q\psi'(q)<\psi(q)+1+\chi\bigr
\} .
\]
Assume moreover that $\chi$ is such that $q_\chi<q_{\max}$. Then, for
all $p$
such that $pq<q_\chi$, it holds that
%
\begin{equation}
\label{eqpsi-pq} 0 < \psi(pq) - p \psi(q) < (p-1) (1+\chi) .
\end{equation}
See Section~\ref{seclemmata} for a proof.
%
\begin{example}
Consider the Poisson cascade introduced by Barral and
Mandelbrot \cite
{barralmandelbrot2002}. Let
$N$ be a Poisson point process with intensity measure $\mu$ on
$(-\infty,\infty)\times(0,\infty]$. Let $\Gamma_i$, $i\in\mathbb Z$
denote the
points of $N$ and let $\{W,W_i\}$ be a collection of i.i.d. positive random
variables such that $\esp[W]=1$. Define the random measure $P$ by
\[
P(A) = \sum\log(W_i) \mathbf1_{\{\Gamma_i\in A\}}
\]
for all relatively compact Borel sets $A \in(-\infty,\infty) \times
(0,\infty]$. Then~(\ref{eqlaplace}) holds with $\psi(q) = \esp[W^q]-1$ and
\begin{eqnarray*}
q_{\max} = \max\bigl\{q\dvt \esp\bigl[W^q\bigr] \leq q\bigr
\} ,\qquad q_{\chi} = \max\bigl\{q\dvt q\esp\bigl[W^q\bigl(
\log(W)-1\bigr)\bigr] \leq1+\chi\bigr\} .
\end{eqnarray*}
\end{example}

\begin{example}
The random measure $P$ can be a Gaussian random measure. Then $P(A)
\sim
\mathbf{N}(-\sigma^2\mu(A)/2,\sigma^2\mu(A))$ and $\psi(q) = \sigma^2
q(q-1)/2$ so that we get the same values of $q_{\max}$, $q_0$ and
$q_\chi$ as
for the multiplicative cascade of the previous section, up to the
$\log2$
term. Note that in this case, $\operatorname{var}(P(A)) = \psi^{\prime\prime}(0)
\mu(A)$ is finite if and only if $\mu(A)<\infty$.
\end{example}

\begin{example}
Let $\alpha\in(0,1)$ and $P$ be a totally skewed to the left $\alpha$-stable
random measure, that is, $\psi(q) = \sigma^\alpha(q-q^\alpha)$. Then
$q_{\max}>1$
if and only if $\sigma^\alpha(1-\alpha)<1$ and then $q_{\max}=\infty$
and for
$\chi\geq0$, $q_\chi= \sigma^{-1} ((1+\chi)/(1-\alpha))^{1/\alpha}$.
It is
noteworthy that, contrary to the previous case, we have here that
$\esp[|P(A)|] = \infty$ and $\esp[\mathrm{e}^{qP(A)}] = \infty$ for
all $A$
such that $\mu(A)>0$ and for all $q<0$, though $\esp[|P(A)|^p] =
c_{p,\alpha}
\sigma^p \mu(A)^{p/\alpha}$ if $p<\alpha$ and $\mu(A)<\infty$.
\end{example}

\begin{example}
Let $\alpha\in(1,2)$ and $P$ be a totally skewed to the left
$\alpha$-stable random measure, that is, $\psi(q) =
\sigma^\alpha(q^\alpha-q)$. Then $q_{\max}>1$ if and only if
$\sigma^\alpha(\alpha-1)<1$ and then $q_{\max}<\infty$. For
$\chi\geq0$, $q_\chi= \sigma^{-1} ((1+\chi)/(\alpha-1))^{1/\alpha}$.
\end{example}

Define, as in the previous section, $L=[2^{n\chi}]$, $\Delta_{k,n}^{(j)} =
[(j+k2^{-n})T,(j+(k+1)2^{-n})T]$ and
\begin{eqnarray*}
S_{L,n}(M,q) &=& \sum_{j=0}^{L-1}
\sum_{k=0}^{2^n-1} M^q\bigl(
\Delta_{k,n}^{(j)}\bigr) ,
\\
\tilde{\zeta}_M(q) &=& 1 + \log_{2} \biggl(
\frac{S_{L,n}(M,q)}{S_{L,n+1}(M,q)} \biggr) .
\end{eqnarray*}

\textit{Consistency}.
For convenience, denote $\tau(q) = \zeta(q)-1$. We have the following result,
whose proof is in Section~\ref{seclemmata}.

\begin{prop} \label{propconv-as-SnMq}
For $q<{q}_\chi$,
\[
L^{-1} 2^{n\tau(q)} S_{L,n}(M,q) \to m(q) \qquad\mbox{a.s. }
\]
\end{prop}
Plugging this into the definition of $\tilde\zeta_M(q)$ yields the
consistency of $\tilde\zeta_M(q)$.
%
\begin{cor}
\label{coroconv-as-tildetau}
For $q<{q}_\chi$,
\[
\tilde{\zeta}_M(q) \to\zeta(q)\qquad \mbox{a.s. }
\]
\end{cor}

\textit{Central limit theorem}.
We next give a central limit theorem for $\tilde{\zeta}_M(q)$ in the mixed
asymptotic framework. Define the centered random variables
%
\begin{eqnarray}
\label{eqdef-Djknq} D_{j,k,n,q} := M^{q}\bigl(
\Delta_{k,n}^{(j)}\bigr) -2^{\tau(q)}
\bigl(M^{q}\bigl(\Delta_{2k,n+1}^{(j)}\bigr) +
M^{q}\bigl(\Delta_{2k+1,n+1}^{(j)}\bigr)\bigr)
\end{eqnarray}
and $D_{j,n,q} = \sum_{k=0}^{2^{n}-1} D_{j,k,n,q}$. By construction, the
variables $D_{j,k,n,q}$ are centered, and stationary and 2-dependent with
respect to $j$. We will start by proving a central limit theorem for $
(L\esp[D_{0,n,q}^{2}])^{-1/2} \sum_{j=0}^{L-1} D_{j,n,q}$. Since the random
variables $D_{j,n,q}$, $0\leq j \leq L-1$, are 2-dependent, it suffices
to show
that for some $p>1$,
%
\begin{equation}
\label{eqdelta1} \lim_{n\to\infty} \frac{ L^{1-p} \esp[ D_{0,n,q}^{2p}] } {
(\esp[D_{0,n,q}^{2}])^p} = 0 .
\end{equation}
We will need the order of magnitude of $D_{0,n,q}$. Set
\[
d_{q} = \esp \bigl[  M^{q}\bigl([0,T] - 2^{\tau(q)}
\bigl\{M^{q}\bigl([0,T/2]\bigr)+M^{q}\bigl([T/2,T]\bigr)
\bigr\} \bigr)^{2} \bigr]
\]
and $d_{k,q} = 2^{n\zeta(2q)} \esp[D_{0,0,n,q}D_{0,k,n,q}]$. By the scaling
property, $\esp[D_{0,0,n,q}^{2}] = 2^{-n\zeta(2q)} d_{q}$ and $d_{k,q}$
does not
depend on $n$. Then,
\[
\esp\bigl[D_{0,n,q}^{2}\bigr] = 2^{-n\tau(2q)}
d_{q} + 2 \cdot2^{-n\tau(2q)} \sum_{k=1}^{2^{n}-1}
\bigl(1-k2^{-n}\bigr) d_{k,q} .
\]
By Lemma~\ref{lemappcov-differences}, we have $d_{k,q} =
\mathrm{O}(k^{-\{\psi(2q)-2\psi(q)+1\}})$. Since $\psi(2q)-2\psi(q)>0$, this implies
that the series $\sum|d_{k,q}|$ is convergent, so the Cesaro\vspace*{1pt} mean
above has a
finite limit and thus $\lim_{n\to\infty} 2^{n\tau(2q)} \esp
[D_{0,n,q}^2] = d_{q}
+ 2 \sum_{k=1}^\infty d_{k,q}$. By
Lemma~\ref{lemappmoment-2p-Dnq} we have $ \esp[D_{0,n,q}^{4}] =
\mathrm{O}(n2^{-n\tau(4q)}+2^{-2n\tau(2q)})$. If $4q<q_\chi$, then
$\psi(4q)-2\psi(2q)<1+\chi$, thus~(\ref{eqdelta1}) holds for $p=2$. The above
discussion leads to the following result.
%
\begin{prop}
\label{propclt-differences-mrm}
If $4q< q_{\chi}$, then there exists a constant $\Theta_q$ such that
\[
L^{-1/2} 2^{-n\tau(2q)/2} \sum_{j=0}^{L-1}
D_{j,n,q} \to_d \mathbf{N}(0,\Theta_q) .
\]
\end{prop}
We can now prove the asymptotic normality of $\tilde\zeta_M(q)$. Denote
\begin{eqnarray*}
R_n & =& \frac{S_{L,n}(M,q) - 2^{\tau(q)} S_{L,n+1}(M,q)} {
S_{L,n}(M,q)}
\\
& =& 2^{n\{2\psi(q)-\psi(2q)-2\psi(q)-1-\chi\}/2} \frac
{L^{-1/2}2^{-n\tau(2q)/2} \sum_{j=0}^{L-1} D_{j,n,q}} {
L^{-1}2^{-n\tau(q)}S_{L,n}(M,q)} .
\end{eqnarray*}
By~(\ref{eqpsi-pq}) applied with $p=2$ and $2q<q_\chi$, it holds that
$1+\chi+2\psi(q)-\psi(2)>0$. Thus, by Propositions~\ref{propconv-as-SnMq}
and~\ref{propclt-differences-mrm}, we have that $R_n=\mathrm{o}(1)$ a.s., so a second
order Taylor expansion yields
\begin{eqnarray*}
\tilde{\zeta}_M(q) - \zeta(q) = \log_2 \biggl(
\frac
{S_{L,n}(M,q)}{2^{\tau(q)} S_{L,n+1}(M,q)} \biggr) = -\log(1-R_n) = R_n +
\mathrm{O}_P\bigl(R_n^2\bigr) .
\end{eqnarray*}
Applying Propositions~\ref{propconv-as-SnMq} and~\ref{propclt-differences-mrm}
yields the next result.
%
\begin{theo}
\label{theocormrm}
If $4q<{q}_\chi$, then
\[
2^{n(1+\chi-\psi(2q)+2\psi(q))/2}\bigl(\tilde{\zeta}_M(q)-\zeta(q)\bigr) \to
\mathbf{N} \bigl( 0, m^{-1}(q)\Theta_q\bigr) .
\]
\end{theo}
For $q,q'<4q_\chi$, it can be shown that $2^{n(1+\chi)} (
2^{n\{2\psi(q)-\psi(2q)\}/2} (\tilde{\zeta}_M(q) - \zeta(q))$,\break
$2^{n\{2\psi(q')-\psi(2q')\}/2} (\tilde{\zeta}_M(q') - \zeta(q')))$ converges
to a bivariate Gaussian distribution with dependent components. The
same comment
holds for the results of the next section.

\section{Multifractal random walk}
\label{secmrw}
Throughout this section, the MRM $M$ and the process $\{w_l(u)\}$ will
be as
defined in the previous section. A multifractal random walk (MRW) is
the process
$X$ obtained as the $L^2$ limit as $l\to0$ of the integral $\int_0^t
\mathrm{e}^{w_l(u)} \,\mathrm{d} B_H(u)$ where $B_H$ is a standard fractional
Brownian motion independent of $M$; see Abry \textit{et~al.} \cite
{abrychainaiscoutinpipiras2009}, Bacry, Delour and
Muzy \cite{bacrydelourmuzy2001}, Bacry and Muzy \cite{bacrymuzy2003}, Lude{\~n}a \cite{ludena2008}.
Recall that $B_H$ is a
continuous centered Gaussian process with $B_H(0)=0$ and
\[
\operatorname{var}\bigl(B_H(t)-B_H(s)
\bigr)=|t-s|^{2H}
\]
for all $t,s\in[0,1]$. For $H=1/2$, $B_{1/2}$ is the standard Brownian motion
and will be simply denoted by $B$. Thus, $X$ is the conditionally (with respect
to $M$) Gaussian process whose covariance function is defined in (\ref{eqmrw1})
or (\ref{eqmrw2}) below according to whether the Hurst parameter of the
fBm is
$H=1/2$ or $H>1/2$. Except for the case $H=1/2$, which is ordinary Brownian
motion, it is worthwhile to remark that this conditionally Gaussian
process $X$
is not the time changed process $B_H(M[0,t])$.

Throughout this section $\to_M$ will stand for conditional convergence in
distribution given $M$ and $\esp_{M}$ and $\operatorname{var}_M$ stand for the
conditional expectation and variance given~$M$. We consider the
following two
cases.
\begin{itemize}
\item Case $H=1/2$ Bacry, Delour and
Muzy \cite{bacrydelourmuzy2001}, Bacry and Muzy \cite{bacrymuzy2003}. The MRW $X$
is defined as the centered, conditionally Gaussian process with conditional
covariance
%
\begin{equation}
\label{eqmrw1} \Gamma(s,t) = \lim_{l\to0+} \int_0^{t\wedge s}
\mathrm{e}^{w_l(u)} \,\mathrm{d} u = M(s \wedge t) .
\end{equation}
The scaling function is $\zeta_{1/2}(q)=\zeta(q/2)$, since
by~(\ref{eqcdv-mrm}) and~(\ref{eqscaling-mrm}), for $\lambda\in(0,1)$,
\begin{eqnarray*}
\bigl\{X(\lambda t), 0 \leq t \leq T\bigr\} & \stackrel{\mathrm{law}}{=}&
\lambda^{1/2} \mathrm e^{\Omega_\lambda/2} \bigl\{X(t) , 0 \leq t \leq T\bigr
\} ,
\\
\esp\bigl[\bigl|X(t)\bigr|^q\bigr] &=& \esp\bigl[\esp_M
\bigl[\bigl|X(t)\bigr|^q\bigr]\bigr]  = c_q \esp
\bigl[M^{q/2}(t)\bigr] = c_q m(q/2) t^{\zeta(q/2)} ,
\end{eqnarray*}
where $c_q = \esp[|\mathbf{N}(0,1)|^q]$ and $m(q)= \esp[M^{q}([0,1])]$.
\item Case $H>1/2$ Abry \textit{et~al.} \cite
{abrychainaiscoutinpipiras2009}, Lude{\~n}a \cite{ludena2008}, Muzy and Bacry \cite{muzybacry2002}. The MRW $X$
is defined as the centered, conditionally
Gaussian process with conditional covariance
%
\begin{eqnarray}
\label{eqmrw2} \Gamma_H (s,t) = \lim_{l\to0+} \int
_0^{t } \int_0^s
\frac{\mathrm
e^{w_l(u)} \mathrm{e}^{w_l(v)}}{|u-v|^{2-2H}} \,\mathrm{d} u\, \mathrm d v = \int_0^t
\int_0^s \frac{M(\mathrm{d} u) M(\mathrm{d}
v)}{|u-v|^{2-2H}} .
\end{eqnarray}
This process is well defined whenever
$H-\psi(2)/2>1/2$, cf. Lude{\~n}a \cite{ludena2008}. Convexity of $\psi$
yields $\psi(2)>0$. The scaling function $\zeta_H$ is defined~by
\[
\zeta_H(q) = qH - \psi(q) ,
\]
since by~(\ref{eqmrw2}) and~(\ref{eqcdv-mrm}) we have
\begin{eqnarray*}
\bigl\{X(\lambda t), 0 \leq t \leq T\bigr\} &
\stackrel{\mathrm{law}}{=}&
\lambda^{H} \mathrm e^{\Omega_\lambda} \bigl\{X(t) , 0 \leq t \leq T\bigr
\} ,
\\
\esp\bigl[\bigl|X(t)\bigr|^q\bigr] & =& c_q m_H(q)
(t/T)^{qH-\psi(q)} ,
\end{eqnarray*}
\end{itemize}
with
%
\begin{equation}
\label{eqdef-mHq} m_H(q) = \esp \biggl[ \biggl\{ \int
_0^T \int_0^T
|u-v|^{2H-2} M(\mathrm {d} u) M(\mathrm{d} v) \biggr\}^{q/2}
\biggr] .
\end{equation}
Since we are considering the mixed asymptotic framework, we assume we
have a
collection of MRM $M^{(j)}$, $j=0,\ldots,L-1$, which are independent, defined
over consecutive intervals of length $T$. For $j=0,\ldots,L-1$ and
$k=0,\ldots,2^{n-1}$, define $\Delta X_{j,k,n}=X_{(j+(k+1)2^{-n})T
}-X_{(j+k2^{-n})T }$. As above, we will investigate the asymptotic properties
of the estimator $\tilde{\zeta}_X(q)$ defined by
\[
\tilde\zeta_X(q) = \log_2 \biggl(\frac
{S_{L,n}(X,q)}{S_{L,n+1}(X,q)}
\biggr) + 1 ,
\]
where now
\[
S_{L,n}(X,q) = \sum_{j=0}^{L-1}
\sum_{k=0}^{2^n -1}| \Delta X_{j,k,n}|^q
.
\]
Denote $\tau_H = \zeta_H(q)-1$ and $T_n(X,q)= S_{L,n}(X,q) - 2^{\tau
_H(q)} S_{L,n+1}(X,q)$. Then
\[
\tilde\zeta_X(q) - \zeta_H(q) = - \log \biggl(1-
\frac
{T_n(X,q)}{S_{L,n(X,q)}} \biggr) .
\]
We will prove that $T_n(X,q)/S_{L,n}(X,q)\to0$ a.s. so that a
Taylor expansion is valid and yields
\[
\tilde\zeta_X(q) - \zeta_H(q) = \frac{T_n(X,q)}{S_{L,n(X,q)}}
\bigl(1+\mathrm{o}(1)\bigr) .
\]
In order to study the ratio above, we will first prove that if $H=1/2$, then
\[
L^{-1} 2^{n\tau(q/2)} S_{L,n}(X,q) \to c_q
m(q/2) ,
\]
and if $H>1/2$ then,
\[
L^{-1} 2^{n\tau_H(q)} S_{L,n}(X,q) \to c_q
m_H(q) ,
\]
with $m_H(q)$ as in~(\ref{eqdef-mHq}) and $c_q=\esp[| \mathbf
{N}(0,1)|^q]$ in
both cases. To study $T_n(X,q)$, we write
\[
T_n(X,q) = T_n(X,q) - \esp_M
\bigl[T_n(X,q)\bigr] + \esp_M\bigl[T_n(X,q)
\bigr] .
\]
We will prove that in both cases, $T_n(X,q) - \esp_M[T_n(X,q)]$ and
$\esp_M[T_n(X,q)]$ converge jointly to independent centered Gaussian
distributions with the same normalization. This will yield the
asymptotic normality of $ \tilde\zeta_X(q) - \zeta_H(q)$.
Because of the different nature of the conditional dependence
structure, which yields different scaling functions, we will consider
the cases $H=1/2$ and $H>1/2$ separately.

\subsection{The case $H=1/2$}
In this case, it holds that
\begin{eqnarray*}
\esp_M \bigl[S_{L,n}(X,q)\bigr]& =& c_q
S_{L,n}(M,q/2) ,
\\
\var_M\bigl(S_{L,n}(X,q)\bigr)& =& \sigma_q^2
S_{L,n}(M,q) ,
\end{eqnarray*}
where $\sigma_q^2 = \var(|\mathbf{N}(0,1)|^q)$.
By Proposition~\ref{propconv-as-SnMq}, if $q<{q}_\chi$, we get
\begin{eqnarray*}
L^{-1} 2^{n\tau(q/2)} \esp_M\bigl[S_{L,n}(X,q)
\bigr]& \to &c_{q} m(q/2) \qquad\mbox{a.s.},
\\
L^{-1} 2^{n\tau(q)} \operatorname{var}_M
\bigl(S_{L,n}(X,q)\bigr) &\to&\sigma_q^2 m(q)\qquad
\mbox{a.s. }
\end{eqnarray*}
This implies that $L^{-1} 2^{n\tau(q/2)} S_{L,n}(X,q)$ converges in probability
to $c_q m(q/2)$. Since $S_{L,n}(X,q)$ is the sum of $L2^n$ conditionally
independent terms, by an application of Borel--Cantelli's lemma similar
to the
one used in the proof of Proposition~\ref{propconv-as-SnMq},
almost sure convergence also holds, that is,
%
\begin{eqnarray}
\label{eqconv-SnqX} L^{-1}2^{n\tau(q/2) } S_{L,n}(X,q) \to
c_{q} m(q/2) \qquad\mbox{a.s.}
\end{eqnarray}
Using the notation~(\ref{eqdef-Djknq}) of the previous section, we have
\begin{eqnarray*}
\esp_M\bigl[T_n(X,q)\bigr]  = c_q
2^{\tau(q/2)}S_{L,n+1}(M,q/2) - c_q S_{L,n}(M,q/2)
= - c_q \sum_{j=0}^{L-1} \sum
_{k=0}^{2^n-1} D_{j,k,n,q} .
\end{eqnarray*}
Thus, by Proposition~\ref{propclt-differences-mrm}, if $q<{q}_\chi$ then
$L^{-1/2} 2^{n \tau(q)/2} \esp_M[T_n(X,q)]$ converges to a centered Gaussian
random variable with variance $\Sigma(1/2,q)$, say.
By the conditional independence of $B$ and $M$, $T_n(X,q)-\esp_M[T_n(X,q)]$ is
a sum of centered and conditionally independent random variables with
conditional variance
\begin{eqnarray*}
\operatorname{var}_M\bigl(T_n(X,q)\bigr) =
\sigma_q^2 S_{L,n}(M,q) + \sigma_q^2
\bigl(2^{2\tau(q/2)}-2^{\tau(q/2)+1}\bigr) S_{L,n+1}(M,q) .
\end{eqnarray*}
By Proposition~\ref{propconv-as-SnMq}, $L^{-1}2^{n\tau(q)}
\operatorname{var}_M(T_n(X,q))$ converges almost surely to the positive constant
$\Gamma(1/2,q)$ defined by
\[
\Gamma(1/2,q) = \sigma_q^2 m(q) \bigl\{1+
\bigl(2^{2\tau(q/2)}-2^{\tau(q/2)+1}\bigr)2^{-\tau(q)} \bigr\} .
\]
Thus,
%
\begin{equation}
\label{eqconv-weak-M} L^{-1/2} 2^{n\tau(q)/2} \bigl
\{T_n(X,q)-\esp_M\bigl[T_n(X,q)\bigr]\bigr\}
\to_M \mathbf{N}\bigl(0,\Gamma(1/2,q)\bigr) .
\end{equation}
Since the variance is deterministic, this assures unconditional
convergence to
the stated Gaussian random variable. Moreover, the conditional
independence of $B$ and $M$
also implies that the sequence of random vectors
\[
L^{-1/2} 2^{n\tau(q)/2} \bigl( T_n(X,q)-
\esp_M\bigl[T_n(X,q)\bigr], \esp_M
\bigl[T_n(X,q)\bigr] \bigr)
\]
converges weakly to $(Z_1,Z_2)$ where $Z_1$ and $Z_2$ are independent Gaussian
random variables with zero mean and variance $\Gamma(1/2,q)$ and
$\Sigma(1/2,q)$, respectively. The previous considerations yield the central
limit theorem for $\tilde\zeta_X(q)$.
%
\begin{theo}
If $q<{q}_\chi$, then
\begin{eqnarray*}
L^{1/2} 2^{n(\psi(q/2)-\psi(q)/2+1/2)} \bigl\{\tilde{\zeta}_X(q) -
\zeta_{1/2}(q)\bigr\} \to_d \mathbf{N} \biggl( 0,
\frac{\Gamma(1/2,q)+\Sigma(1/2,q)} {c_q^2 m^2(q/2)} \biggr) .
\end{eqnarray*}
\end{theo}

\subsection{Case $H>1/2$}
\label{secH>12}
We begin by studying
$\esp_M[T_n(X,q)]$. Define $a_{j,k,n,H} = \esp_{M}^{1/2}[(\Delta
X_{j,k,n})^2]$. Then
\begin{eqnarray*}
\esp_M\bigl[T_n(X,q)\bigr] = c_q \sum
_{j=0}^{L-1} \sum
_{k=0}^{2^n-1} \bigl( 2^{\tau_H(q)} \bigl
\{a_{j,2k,n+1,H}^q + a_{j,2k+1,n+1,H}^q \bigr\} -
a_{j,k,n,H}^q \bigr) .
\end{eqnarray*}
Denote $U_{j,k,n} = 2^{\tau_H(q)} \{a_{j,2k,n+1,H}^q +
a_{j,2k+1,n+1,H}^q \} -
a_{j,k,n,H}^q$ and define $U_{j,n}:=\break \sum_{k=0}^{2^n-1} U_{j,k,n}$.\vspace*{1pt} Then the
collection $\{U_{j,n}\}_{0\le j\le L-1}$ is centered, 2-dependent and
identically distributed. Remark that $\vartheta(q) = 2^{n\zeta_H(2q)}
\operatorname{var}(U_{j,k,n})$ depends only on $q$. By stationarity, for
\mbox{$j=0,\ldots,L-1$},
\begin{eqnarray*}
\operatorname{var} (U_{j,n} ) = 2^{-n\tau_H(2q)} v_q + 2
2^{-n\tau_H(q)} \sum_{k=1}^{2^n-1}
\bigl(2^n-k\bigr) 2^{n\zeta_H(2q)} \operatorname{cov}(U_{0,0,n},U_{0,k,n})
.
\end{eqnarray*}
By Lemma~\ref{lemappcov-diff-H>12}, $2^{n\zeta_H(2q)}
|\operatorname{cov}(U_{0,n,0},U_{0,n,k}) | \leq C k^{-\{\psi(2q)-2\psi(q)+1\}
}$. This
series is convergent, thus the Cesaro mean above converges to its sum. Arguing
as in the proof of Proposition~\ref{propclt-differences-mrm}, in order
to prove
the central limit theorem for $\esp_M[T_n(X,q)]$, since the centered random
variables $U_{j,n}$, $0\leq j \leq L-1$, are 2-dependent, it suffices
to show
that
\[
\lim_{n\to\infty} \frac{ L^{1-p} \esp[ U_{0,k,n}^{4}] } {
(\esp[U_{0,k,n}^{2}])^2} = 0 .
\]
This is done as in Lemma~\ref{lemappmoment-2p-Dnq} using
Lemma~\ref{lemappcov-diff-H>12}. We then have the following result.
%
\begin{prop}
\label{prophmayor-1}
If $2q<q_\chi$, there exists a positive constant $\Sigma(H,q)$ such that
\[
L^{-1} 2^{-n\tau_H(2q)} \operatorname{var}\bigl(\esp_M
\bigl[T_n(X,q)\bigr]\bigr)  \to\Sigma (H,q) .
\]
Moreover, if $4q<q_\chi$, then
%
\begin{equation}
\label{eqconv-weak-M-H-1} L^{-1/2} 2^{-n\tau_H(2q)/2}
\esp_M\bigl[T_n(X,q)\bigr]  \to_M
\mathbf{N}\bigl(0,\Sigma(H,q)\bigr).
\end{equation}
\end{prop}

We next need a result which parallels~(\ref{eqconv-weak-M}). Its proof
is more
involved and is postponed to Section~\ref{seclemmata}.

\begin{prop}
\label{prophmayor}
Let $H<3/4$. If $2q<q_\chi$, then there exists a positive constant
$\Gamma(H,q)$ such that
%
\begin{equation}
\label{eqconv-var-TLnXq-H} L^{-1} 2^{n\tau_H(2q)}
\operatorname{var}_M\bigl(T_{L,n}(X,q)\bigr) \to
\Gamma(H,q) \qquad\mbox{a.s. }
\end{equation}
and if {$4q<q_\chi$}, then
%
\begin{eqnarray}
\label{eqconv-weak-M-H-2} L^{-1/2} 2^{n\tau_H(2q)/2} \bigl
\{T_n(X,q)-\esp_M\bigl[T_n(X,q)\bigr]\bigr\}
\to_M \mathbf{N}\bigl(0,\Gamma(H,q)\bigr) .
\end{eqnarray}
\end{prop}
As for the case $H=1/2$, the fact that $\Gamma(H,q)$ is deterministic
establishes unconditional convergence in distribution. The proof
of~(\ref{eqconv-weak-M-H-2}) is based on the recent results of
Nualart and
Peccati~\cite{nualartpecatti2005} on the convergence of sequences of random variables
in a Gaussian chaos. Altogether, (\ref{eqconv-weak-M-H-1})
and~(\ref{eqconv-weak-M-H-2}) yield the asymptotic normality of the estimator.
%
\begin{theo}
\label{theo12<H<34}
If $4q<{q}_\chi$ and $H<3/4$, then
\begin{eqnarray*}
2^{n(1+\chi-\psi(2q)+2\psi(q))/2} \bigl\{\tilde\zeta_X(q) - \zeta_H(q)
\bigr\} \to_d \mathbf{N} \biggl(0, \frac{\Gamma(H,q) + \Sigma(H,q)}{c_q^2 m_H^2(q)} \biggr) .
\end{eqnarray*}
\end{theo}

\section{Proofs}
\label{seclemmata}
In all the proofs, without loss of generality, we set $T=1$. We start by
proving~(\ref{eqpsi-pq}). The convexity of $\psi$ and $\psi(1)=0$
implies that
$q_{\max}>1$ if and only if $\psi'(1)<1$, and $\psi'(q_{\max})>1$. This
in turn
implies that $1<q_0<q_{\max}$.
The convexity of $\psi$ also implies that the function $q\mapsto
q\psi'(q)-\psi(q)$ is increasing, thus $q_\chi>q_0$ for all $\chi>0$.\vadjust{\goodbreak} Consider
the positive and increasing function $p\mapsto\psi(pq) - p\psi(q)$. By
convexity, for $p>1$, $\psi(pq) - \psi(q) \leq\psi'(pq)(pq-p)$. This yields,
for $p>1$ and $pq<q_\chi$,
\begin{eqnarray*}
0& < &\psi(pq) - p\psi(q)  = p\psi(pq) - p\psi(q) - (p-1) \psi(pq)
\\
& \leq&(p-1) \bigl\{pq\psi'(pq) - \psi(pq)\bigr\} < (p-1) (1+\chi)
.
\end{eqnarray*}
This proves~(\ref{eqpsi-pq}).

We will also repeatedly use an argument of $m$-dependence. If
$\xi_1,\ldots,\xi_N$ are $m$-dependent random variables with zero mean
and finite
stationary $p$th moment, $1 \leq p \leq2$, then there exists a
constant $C$
which depends only on $p$ such that
%
\begin{equation}
\esp \Biggl[ \Biggl\llvert \sum_{i=1}^N
\xi_i \Biggr\rrvert^p \Biggr] \leq C m^{p-1} N
\esp\bigl[|\xi_1|^p\bigr] . \label{eqm-dep}
\end{equation}

\subsection{\texorpdfstring{Proof of Proposition~\protect\ref{propconv-as-SnMq}}{Proof of Proposition 3.1}}
Let $n_0\geq2$ be an integer, $\alpha=1/n_0$ and $l_n=2^{-(1-\alpha) n}$.
Fix $q<q_\chi$. We can choose $\alpha<\chi$ small enough so that
$q<q_{\chi'}$
with $\chi'<\chi-\alpha$. Then, we can also choose $p>1$, close enough
to 1,
such that $pq<q_{\chi'}$ and without loss of generality, we can also
impose that
$p-1<\alpha(q\vee1)/2$. Define
%
\begin{eqnarray}
\label{eqdef-tilde-tnq} \tilde T_{n,q} &=& 2^{-n}
L^{-1} \sum_{j=0}^{L-1} \sum
_{k=0}^{2^n-1} \frac{\mathrm{e}^{qw_{l_n}(j+2^{-n}k)}}{\esp[\mathrm
{e}^{qw_{l_n}(0)}]}
\nonumber
\\[-8pt]
\\[-8pt]
\nonumber
&=&
2^{-n} L^{-1} \mathrm e^{-\psi(q)} l_n^{\psi(q)}
\sum_{j=0}^{L-1} \sum
_{k=0}^{2^n-1} \mathrm e^{qw_{l_n}(j+2^{-n}k)} .
\end{eqnarray}
We will prove that for $\alpha$ and $p>1$ chosen as above, there exist constants
$C, \eta>0$ such that
%
\begin{eqnarray}
\esp \bigl[ \llvert \tilde{T}_{n,q} - 1\rrvert^{p} \bigr]
&\leq& C 2^{-n\eta} , \label{eqfirstpart}
\\
\esp \biggl[ \biggl\llvert \tilde{T}_{n,q} - \frac{S_{L,n}(M,q)} {\esp[S_{L,n}(M,q)]} \biggr
\rrvert^{p} \biggr]&\leq & C 2^{-n\eta} . \label{eqsecondpart}
\end{eqnarray}
The above inequalities and an application of Borel--Cantelli's lemma
yield that
$\tilde{T}_{n,q} \to1$, a.s. and
\[
\frac{S_{L,n}(M,q)}{\esp[S_{L,n}(M,q)]} - \tilde{T}_{n,q}  \to0 \qquad\mbox {a.s.}
\]
For all $j,k,n$, we have $\esp[M^q(\Delta_{k,n}^{(j)})] = 2^{-n\zeta(q)}m(q)$,
so that $\esp[S_{L,n}(M,q)] = L 2^{-n\tau(q)} m(q)$. Thus,
Proposition~\ref{propconv-as-SnMq} follows.

\begin{pf*}{Proof of~(\ref{eqfirstpart})}
Define $\varepsilon=p-1$. The variables $\mathrm
{e}^{qw_{l_n}(j+2^{-n}k)}-\esp[\mathrm
e^{qw_{l_n}(j+2^{-n}k)}]$ are $2$-dependent (in $j$) and centered, so there
exists a constant $C>0$ such that
\[
\esp \bigl[ \llvert \tilde{T}_{n,q}-1 \rrvert^{p} \bigr]
\leq \frac{C}{L^{\varepsilon}} \esp \Biggl[ \Biggl\llvert \frac{1}{2^n} \sum
_{k=0}^{2^n-1} \frac{\mathrm{e}^{qw_{l_n} (2^{-n}k)}}{\mathrm{e}^{\psi(q)} l_n^{-\psi(q)}} - 1 \Biggr
\rrvert^p \Biggr] .
\]
By Lemma~\ref{lemappmerida}, for any $\varepsilon'<\varepsilon$, there
exists a constant $C$ such that
\[
\esp \bigl[ \llvert \tilde{T}_{n,q}-1 \rrvert^{p} \bigr]
\leq C 2^{n\{
(1-\alpha) \{\psi(pq) - p\psi(q) - \varepsilon'\}-\varepsilon\chi\}} .
\]
By~(\ref{eqpsi-pq}), since $pq<q_{\chi'}<q_\chi$, we have
\begin{eqnarray*}
(1-\alpha) \bigl\{\psi(pq) - p \psi(q) - \varepsilon'\bigr\} -
\varepsilon\chi & < &(1-\alpha)\bigl\{\varepsilon\bigl(1+\chi'\bigr)-
\varepsilon'\bigr\} -\varepsilon\chi
\\
& <& (1-\alpha)\bigl\{\varepsilon(1+\chi-\alpha)-\varepsilon'\bigr\} -
\varepsilon\chi
\\
& =& (1-\alpha)\bigl\{\varepsilon(1-\alpha)-\varepsilon'\bigr\} -\alpha
\varepsilon\chi .
\end{eqnarray*}
This can be made negative by choosing $\varepsilon'>(1-\alpha)\varepsilon$.
\end{pf*}
\begin{pf*}{Proof of~(\ref{eqsecondpart})}
We start by using again the argument of 2-dependence in $j$, to
obtain, for
some constant $C$,
%
\begin{eqnarray}
\esp \biggl[ \biggl\llvert \frac{S_{L,n}(M,q)} {\esp[S_{L,n}(M,q)]} - \tilde T_{n,q} \biggr
\rrvert^p \biggr] \leq\frac C {L^\varepsilon} \esp \Biggl[\Biggl
\llvert \frac1{2^n} \sum_{k=0}^{2^n-1}
\frac{M^q(\Delta_{k,n}^{(0)})}{2^{-n\zeta(q)}m(q)} - \frac{\mathrm{e}^{qw_{l_n}(k2^{-n})}}{\mathrm{e}^{\psi(q)}
l_n^{-\psi(q)} } \Biggr\rrvert^{p} \Biggr] .
\label{eq2-dep}
\end{eqnarray}
For clarity, we now omit the superscript $(0)$ in $\Delta_{k,n}^{(0)}$. Let
$M_n$ denote the random measure with density $\mathrm{e}^{-w_{l_n}}$ with
respect to $M$. By construction, the measure $M_n$ is independent of the
process~$w_{l_n}$. Indeed, for any Borel set $A$, $M_n(A) = \lim_{l\to
0} \int_A
\mathrm e^{w_l(u)-w_{l_n}(u)} \,\mathrm{d} u$, and for $l<l_n$, $w_l-w_{l_n}$
is independent of $w_{l_n}$, by the independent increment property of
the random
measure~$P$. Denote
\[
\tilde S_n = 2^{n\tau(q)} \sum_{k=0}^{2^n - 1}
\mathrm{e}^{qw_{l_n}(k2^{-n})} M_n(\Delta_{k,n}) .
\]
Applying the bound~(\ref{eqappcota16}) in Lemma~\ref
{lemappapprox-MnDeltan}, we obtain
\begin{eqnarray*}
\esp\Biggl[\Biggl|\tilde S_n - 2^{n\tau(q)} \sum
_{k=0}^{2^n-1} M(\Delta_{k,n})\Biggr|^p
\Biggr] \leq C 2^{-n\alpha(q\vee1)/2} 2^{n\{\psi(pq)-p\psi(q)\}} .
\end{eqnarray*}
Since we have chosen $\varepsilon<\alpha(q\vee1)/2$, by~(\ref{eqpsi-pq}),
we have
\[
\psi(pq)-p\psi(q) -\alpha(q\vee1)/2 -\varepsilon\chi< \varepsilon- \alpha (q\vee1)/2
<0 .
\]
Define $m_n(q) = \mathrm{e}^{\psi(q)} l_n^{-\psi(q)} 2^{n\zeta(q)}
\esp[M_n^q(\Delta_{k,n})]$. By~(\ref{eqappmndeltan}), we have $\lim_{n\to\infty}
m_n(q) = m(q)$ and thus $\esp[M_n^q(\Delta_{0,n})] \sim l_n^{\psi(q)}
2^{-n\zeta(q)} \mathrm{e}^{-\psi(q)} m(q)$.\vadjust{\goodbreak} Next, we note that the random
variables $M_n(\Delta_{k,n})$ are $2^n l_n$-dependent and $\mathrm{e}^{w_{l_n}}$
is independent of $M_n$. Thus, applying~(\ref{eqm-dep}) conditionally on
$w_{l_n}$ yields
\begin{eqnarray*}
\esp \Biggl[ \Biggl\llvert \frac{\tilde S_n}{m_n(q)} - 2^{-n} \sum
_{k=0}^{2^n -1} \frac{\mathrm{e}^{qw_{l_n}(k2^{-n})}}{\mathrm{e}^{\psi(q)} l_n^{-\psi(q)}} \Biggr
\rrvert^p \Biggr] & = &\esp \Biggl[ \Biggl\llvert 2^{-n} \sum
_{k=0}^{2^n -1} \frac{\mathrm{e}^{qw_{l_n}(k2^{-n})}}{\mathrm{e}^{\psi(q)} l_n^{-\psi(q)}} \biggl(
\frac{M_n^q(\Delta_{k,n})}{\esp[M_n^q(\Delta_{k,n})]} - 1 \biggr) \Biggr\rrvert^p \Biggr]
\\
& \leq& C l_n^{-\psi(pq)+p\psi(q)-\varepsilon} \esp \biggl[\biggl\llvert
\frac{M_n^q(\Delta_{0,n})} {\esp[M_n^q(\Delta_{0,n})]} \biggr\rrvert^p \biggr]
\\
& \leq& C 2^{n\{\psi(pq)-p\psi(q)-\varepsilon(1-\alpha)\}} .
\end{eqnarray*}
Using the fact that $pq<q_{\chi'}$,~(\ref{eqpsi-pq}) and $\chi'<\chi
-\alpha$, we obtain
\begin{eqnarray*}
\psi(pq)-p\psi(q) - (1-\alpha) \varepsilon-\varepsilon\chi\leq\varepsilon \bigl(1+
\chi'\bigr) - (1-\alpha)\varepsilon-\varepsilon\chi= \varepsilon\bigl(
\chi' + \alpha- \chi\bigr) < 0 .
\end{eqnarray*}
This concludes the proof of~(\ref{eqsecondpart}).\vspace*{-2pt}
\end{pf*}

\subsection{\texorpdfstring{Proof of Proposition~\protect\ref{prophmayor}}
{Proof of Proposition 4.3}}\vspace*{-2pt}
Define $a_{j,k,n,H} = \esp_{M}^{1/2}[(\Delta X_{j,k,n})^2]$ and the
conditionally standard Gaussian random variables
\[
Y_{j,k,n}=\Delta X_{j,k,n}/a_{j,k,n,H} .
\]
Let $G_q(x)=|x|^q-c_q$. With this notation, we have
\[
S_{L,n}(X,q) - \esp_M\bigl[S_{L,n}(X,q)\bigr]
= \sum_{j=0}^{L-1} \sum
_{k=0}^{2^n-1} a_{j,k,n,H}^q
G_q(Y_{j,k,n}) .
\]
Let $g_r(q)$, $r\geq0$, be the coefficients of the expansion of $G_q$
over the
Hermite polynomials $\{H_r,r\geq0\}$ (which are defined in such a way that
$\esp[H_k(X)H_l(X)] = k!$ if $k=l$ and 0 otherwise), that is, $g_r(q) =
\esp[H_r(V)G_q(V)]$ where $V$ is a standard Gaussian random variable. Since
$G_q$ is a centered even function, $g_r(q)=0$ for $r=0,1$. Since
$\esp[G_q^2(X)]<\infty$, the series $\sum_{r=2}g_r^2(q)/r!$ is summable
and $G_q
= \sum_{r=2}^\infty\frac{g_r(q)}{r!} H_r$. Then, by Mehler's formula (see,
e.g., Arcones \cite{arcones1994}), we have
\[
L^{-1} 2^{n\tau_H(2q)} \operatorname{var}_M
\bigl(S_{L,n}(X,q)\bigr) = \sum_{r=2}^\infty
\frac{g_r(q)^2}{r!} \Gamma_n(r,q) ,
\]
with
\begin{eqnarray*}
\Gamma_n(r,q) & =& L^{-1} 2^{n\tau_H(q)}
(r!)^{-1} \operatorname{var}_M \Biggl( \sum
_{j=0}^{L-1}\sum_{k=0}^{2^n-1}
a_{j,k,n,H}^q H_r(Y_{j,k,n}) \Biggr)
\\
& =& L^{-1} 2^{n\tau_H(q)} \sum_{j_1,j_2=0}^{L-1}
\sum_{k,k'=0}^{2^n-1} \rho_{H,n}^r
\bigl(j_1,j_2,k,k'\bigr)
{a}_{j_1,k,n,H}^q {a}_{j_2,k',n,H}^q
\end{eqnarray*}
for $r\in\mathbb{N}, r \ge2$, and the conditional correlations (which
are zero
if $H=1/2$) are
\[
\rho_{H,n}\bigl(j_1,j_2,k,k'
\bigr) = \operatorname{cov}_M( Y_{j_1,k,n}, Y_{j_2,k',n})
= \frac{\esp_M[\Delta X_{j_1,k,n} \Delta X_{j_2,k',n}]}{a_{j_1,k,n,H}
a_{j_2,k',n,H}} .
\]
By Lemma 3.1 in Lude{\~n}a \cite{ludena2008}, for $j_1<j_2$ and $k<k'$, we have
the bound
%
\begin{eqnarray}
\label{eqbound-cov-ludena} \rho_{H,n}\bigl(j_1,j_2,k,k'
\bigr) \leq\min\bigl(1,C \bigl|(j_2-j_1)2^n +
\bigl(k'-k\bigr)\bigr|^{2H-2}\bigr)
\end{eqnarray}
for some deterministic constant $C$. We start by proving that for
$H<3/4$ and
{$2q<q_\chi$}, there exists a constant $\Gamma(r,q)$ such that
%
\begin{equation}
\lim_{n\to\infty} 2^{n(2\psi(q)-\psi(2q)+1+\chi)} \esp\bigl[\Gamma_n(r,q)\bigr]
= \Gamma(r,q) . \label{eqgammm-rq}
\end{equation}
By the scaling property,
\[
\esp\bigl[a_{j,k,n,H}^q\bigr] = 2^{-n\zeta_H(q)}
m_H(q) ,
\]
with $\zeta_H(q) = qH - \psi(q)$. Thus, denoting
$v_\chi(q)=2\psi(q)-\psi(2q)+1+\chi$, by stationarity, we have
%
\begin{eqnarray}
\label{eqesp-gamma-nrq} &&2^{nv_\chi(q)} \esp\bigl[\Gamma_n(r,q)
\bigr] \nonumber\\
&&\quad= m_H(2q) + 2^{-n} 2^{n\zeta_H(2q)} \sum
_{k\ne k'} \esp\bigl[\rho_{H,n}^r
\bigl(0,0,k,k'\bigr) a_{0,k,n,H}^q
a_{0,k',n,H}^q\bigr]
\\
&&\qquad{}+ 2^{-n(1+\chi)} 2^{n\zeta_H(2q)} \sum_{j\ne j'}
\sum_{k,k'} \esp\bigl[\rho_{H,n}^r
\bigl(j,j',k,k'\bigr) a_{j,k,n,H}^q
a_{j',k',n,H}^q\bigr].\nonumber
\end{eqnarray}
Consider the middle term. Recall that
\begin{eqnarray*}
&&\rho_{n,H}^r\bigl(0,0,k,k'\bigr)
a_{0,k,n,H}^q a_{0,k',n,H}^q\\
&&\quad = \biggl\{ \int
_{k2^{-n}}^{(k+1)2^{-n}} \int_{k'2^{-n}}^{(k'+1)2^{-n}}
|u-v|^{2H-2} M(\mathrm{d} u) M(\mathrm{d} v) \biggr\}^r
\\
& &\qquad{}\times \biggl\{ \int_{k2^{-n}}^{(k+1)2^{-n}} \int
_{k2^{-n}}^{(k+1)2^{-n}} |u-v|^{2H-2} M(\mathrm{d} u) M(
\mathrm{d} v) \biggr\}^{(q-r)/2}
\\
&&\qquad{} \times \biggl\{ \int_{k'2^{-n}}^{(k'+1)2^{-n}} \int
_{k'2^{-n}}^{(k'+1)2^{-n}} |u-v|^{2H-2} M(\mathrm{d} u) M(
\mathrm{d} v) \biggr\}^{(q-r)/2} .
\end{eqnarray*}
Assume that $k<k'$ and denote $\ell=k'-k+1$. By the scaling property
and the
stationarity of the increments of $M$, we have
\begin{eqnarray*}
&&\rho_{n,H}^r\bigl(0,0,k,k'\bigr)
a_{0,k,n,H}^q a_{0,k',n,H}^q
\\
& &\quad\stackrel{(\mathrm{law})} = \bigl(\ell2^{-n}\bigr)^{r(2H-2)+2r}
\mathrm{e}^{2r\Omega
_{\ell
2^{-n}}} \biggl\{ \int_{0}^{1/\ell}
\int_{1-1/\ell}^1 |u-v|^{2H-2} M(\mathrm{d}
u) M(\mathrm{d} v) \biggr\}^r
\\
&&\hspace*{5pt}\qquad{} \times\bigl(\ell2^{-n}\bigr)^{(q-r)(H-1)+q-r} \mathrm{e}^{(q-r) \Omega_{\ell
2^{-n}}}
\biggl\{ \int_0^{1/\ell} \int_0^{1/\ell}
|u-v|^{2H-2} M(\mathrm{d} u) M(\mathrm{d} v) \biggr\}^{(q-r)/2}
\\[-2pt]
&&\hspace*{5pt}\qquad{} \times\bigl(\ell2^{-n}\bigr)^{(q-r)(H-1)+q-r} \mathrm{e}^{(q-r) \Omega_{\ell
2^{-n}}}
\biggl\{ \int_{1-1/\ell}^{1} \int_{1-1/\ell}^{1}
|u-v|^{2H-2} M(\mathrm{d} u) M(\mathrm{d} v) \biggr\}^{(q-r)/2}
\\[-2pt]
&&\hspace*{5pt}\quad = \bigl(\ell2^{-n}\bigr)^{2qH} \mathrm{e}^{2q \Omega_{\ell2^{-n}}}
Q_\ell^r a_\ell^q
b_\ell^q ,
\end{eqnarray*}
with
\begin{eqnarray*}
a_\ell^2 & =& \int_0^{1/\ell}
\int_0^{1/\ell} |u-v|^{2H-2} M(\mathrm{d}
u) M(\mathrm{d} v) ,
\\[-2pt]
b_\ell^2 & =& \int_{1-1/\ell}^{1}
\int_{1-1/\ell}^{1} |u-v|^{2H-2} M(\mathrm{d}
u) M(\mathrm{d} v) ,
\\[-2pt]
Q_\ell& =& \frac{\int_0^{1/\ell} \int_{1-1/\ell}^{1} |u-v|^{2H-2}
M(\mathrm{d}
u) M(\mathrm{d} v)}{a_\ell b_\ell} .
\end{eqnarray*}
With this notation, the middle term in~(\ref{eqesp-gamma-nrq}) can be expressed
as
\begin{eqnarray*}
&& 2 \cdot2^{n\zeta_H(2q)}\sum_{\ell=1}^{2^n-1}
\bigl(1-\ell2^{-n}\bigr) \bigl(\ell 2^{-n}
\bigr)^{2qH} \esp\bigl[\mathrm{e}^{2q\Omega_{\ell2^{-n}}}\bigr] \esp
\bigl[Q_\ell^r a_\ell^q
b_\ell^q \bigr]
\\[-2pt]
&&\quad = 2 \cdot2^{n\zeta_H(2q)} 2^{-n\{2qH-\psi(2q)\}} \sum
_{\ell=1}^{2^{n-1}} \bigl(1-\ell2^{-n}\bigr)
\ell^{2qH-\psi(2q)} \esp\bigl[Q_\ell^r
a_\ell^q b_\ell^q \bigr]
\\[-2pt]
&&\quad = 2 \sum_{\ell=1}^{2^{n-1}} \bigl(1-
\ell2^{-n}\bigr) \ell^{\zeta_H(2q)} \esp\bigl[Q_\ell^r
a_\ell^q b_\ell^q \bigr] .
\end{eqnarray*}
Moreover, $a_\ell\geq\ell^{2-2H} M([0,1/\ell])$, $b_\ell\geq\ell^{2-2H}
M([1-1/\ell,1])$, and the numerator in $Q_\ell$ is bounded from above by
$(1-2/\ell)^{2H-2}M([0,1/\ell]) M([1-1/\ell,1])$. Thus,
%
\begin{equation}
Q_\ell\leq C \ell^{2H-2} \label{eqdeterministic-bound}
\end{equation}
for some deterministic constant $C$. This and H\"older's inequality yields
\[
\esp\bigl[Q_\ell^r a_\ell^q
b_\ell^q \bigr] \leq C \ell^{r(2H-2)}
\esp^{1/2}\bigl[a_\ell^{2q}\bigr]
\esp^{1/2}\bigl[b_\ell^{2q}\bigr] .
\]
Applying the scaling property of $M$ yields $\esp[a_\ell^{2q}] =
\esp[b_\ell^{2q}] = \ell^{-\zeta_H(2q)} m_H(q)$, hence
\[
\ell^{\zeta_H(2q)} \esp\bigl[Q_\ell^r
a_\ell^q b_\ell^q \bigr] \leq C
\ell^{r(2H-2)} .
\]
Since $r\geq2$ and $H<3/4$, the series $\ell^{r(2H-2)}$ is summable,
and thus
\[
\lim_{n\to\infty} \sum_{\ell=1}^{2^{n-1}}
\bigl(1-\ell2^{-n}\bigr) \ell^{\zeta_H(2q)} \esp\bigl[Q_\ell^r
a_\ell^q b_\ell^q \bigr] = \sum
_{\ell=1}^{\infty} \ell^{\zeta_H(2q)} \esp
\bigl[Q_\ell^r a_\ell^q
b_\ell^q \bigr].\vadjust{\goodbreak}
\]
Consider now the last term in~(\ref{eqesp-gamma-nrq}), say $RR_n$.
Using the
bound~(\ref{eqbound-cov-ludena}), the scaling property, the fact that the
$a_{j,k,n,H}$ are $2$-dependent, and $H<3/4$, we have
%
\begin{eqnarray}
RR_n \leq C 2^{n\{\zeta_H(2q)-2\zeta_H(q)\}} \sum_{j=1}^L
\sum_{k=1}^{2^n} \bigl(j2^n+k
\bigr)^{2H-2} = \mathrm{O}\bigl(2^{n\{2\psi(q)-\psi(2q)\}}\bigr) = \mathrm{o}(1) . \label{eqbound-RR}
\end{eqnarray}
This proves~(\ref{eqgammm-rq}).
We now prove that if $H<3/4$, for each $r\ge2$,
%
\begin{equation}
\label{eqratio} \Gamma_n(r,q)/\esp\bigl[\Gamma_n(r,q)
\bigr] \to1 \qquad\mbox{a.s.}
\end{equation}
or equivalently
\[
2^{n\{1+\chi-\psi(2q)+2\psi(q)\}} \Gamma_n(r,q)
\to\Gamma(r,q)\qquad \mbox {a.s. }
\]
Write $ 2^{n\{1+\chi-\psi(2q)+2\psi(q)\}}\Gamma_n(r,q) =
S_{n,1}+S_{n,2}+S_{n,3}$ with
\begin{eqnarray*}
S_{n,1} & =& 2^{n\tau_H(2q)} L^{-1} \sum
_{j=0}^{L-1} \sum_{k=0}^{2^n-1}
a_{j,k,n,H}^{2q} ,
\\
S_{n,2} & =& 2^{n\tau_H(2q)} L^{-1} \sum
_{j=0}^{L-1} \sum_{0 \leq
k\ne k' < 2^n}
\rho_{H,n}^r\bigl(j,j,k,k'\bigr)
a_{j,k,n,H}^{q} a_{j,k',n,H}^{q} ,
\\
S_{n,3} & = & 2^{n\tau_H(2q)} L^{-1} \sum
_{0 \leq j \ne j' < L} \sum_{k,k'=0}^{2^n-1}
\rho_{H,n}^r\bigl(j',j',k,k'
\bigr) a_{j,k,n,H}^{q} a_{j',k',n,H}^{q} .
\end{eqnarray*}
The bound~(\ref{eqbound-RR}) and Borel--Cantelli's lemma implies
that $S_{n,3}\to0$ a.s.
Define $\tilde{a}_{j,k,n,H} = \mathrm{e}^{w_{l_n}(t_{j,k})} \tilde
\delta_{j,k,n,H}$ with
\[
\tilde\delta_{j,k,n,H}^2 = \int_{\Delta_{k,n}^{(j)}} \int
_{\Delta_{k,n}^{(j)}} |u-v|^{2H-2} M_n(\mathrm{d} u)
M_n(\mathrm{d} v) .
\]
By Lemma~\ref{lemappcotadef1}, we have, if $2q<q_\chi$,
%
\begin{equation}
\lim_{n\to\infty} 2^{n\zeta_H(2q)} \mathrm{e}^{\psi(2q)}
l_n^{-\psi
(2q)} \esp\bigl[ \tilde\delta_{j,k,n,H}^{2q}
\bigr] = m_H(2q) . \label{eqlim-mh2q}
\end{equation}
By $2$-dependence with respect to $j$, Jensen's inequality, (\ref{eqpsi-pq})
applied to $2q<q_\chi$ and the bound~(\ref{eqappcota164}), we obtain,
some $\eta>0$,
\begin{eqnarray*}
\esp \Biggl[ \Biggl\llvert 2^{n\tau_H(2q)} L^{-1} \sum
_{j=0}^{L-1} \sum_{k=0}^{2^n-1}
\bigl(a_{j,k,n,H}^{2q}-\tilde a_{j,k,n,H}^{2q}
\bigr) \Biggr\rrvert^p \Biggr] & \leq& C L^{1-p}
2^{np\zeta(2q)} \esp \bigl[ \bigl\llvert a_{j,k,n,H}^{2q}-
\tilde a_{j,k,n,H}^{2q} \bigr\rrvert^p \bigr]
\\
& \leq& C L^{1-p} 2^{n\psi(2pq) - p\psi(2q)-\eta} \leq C 2^{n(p-1-\eta)} .
\end{eqnarray*}
Choosing $p-1<\eta$ and Borel--Cantelli's lemma yield that
%
\begin{equation}
2^{n\tau_H(2q)} L^{-1} \sum_{j=0}^{L-1}
\sum_{k=0}^{2^n-1} \bigl(a_{j,k,n,H}^{2q}-
\tilde a_{j,k,n,H}^{2q}\bigr) \to0 \qquad \mbox{a.s.} \label{eqa-atilde-ps}
\end{equation}
Recall the definition of $\tilde T_{n,2q}$ in~(\ref{eqdef-tilde-tnq})
and define further
\begin{eqnarray*}
\tilde S_{n,1} & =& 2^{n\tau_H(2q)} L^{-1} \sum
_{j=0}^{L-1} \sum_{k=0}^{2^n-1}
\tilde a_{j,k,n,H}^{2q} ,
\\
m_{n,H}(2q) & =& 2^{n\zeta_H(q)} \esp\bigl[\tilde
a_{0,0,n,H}^{2q}\bigr] = 2^{n\zeta_H(q)} \mathrm{e}^{\psi(q)}
l_n^{-\psi(q)} \esp\bigl[\tilde\delta_{0,0,n,H}^{2q}
\bigr] .
\end{eqnarray*}
We have already shown in the proof of Proposition~\ref
{propconv-as-SnMq} that
if $2q<q_\chi$, then $\tilde T_{n,2q} \to1$ a.s. Moreover, by the
argument of
$2$-dependence with respect to $j$, we have
\begin{eqnarray*}
\esp \biggl[ \biggl\llvert \frac{\tilde S_{n,1}}{m_{n,H}(2q)} - \tilde T_{n,2q} \biggr
\rrvert^p \biggr]  \leq C L^{1-p} \esp \Biggl[ \Biggl\llvert
2^{-n} \sum_{k=1}^n
\frac
{\mathrm
e^{2qw_{l_n}(k2^{-n})}}{\mathrm{e}^{\psi(2q)} l_n^{-\psi(2q)}} \biggl(\frac{\tilde
\delta_{0,k,n,H}^{2q}}{\esp[\tilde\delta_{0,k,n,H}^{2q}]} - 1 \biggr) \Biggr
\rrvert^p \Biggr] .
\end{eqnarray*}
As in the proof of Proposition~\ref{propconv-as-SnMq}, we now use the
fact that
$w_{l_n}$ is independent of the measure~$M_n$, the $2^nl_n$-dependence
of the
variables $\tilde\delta_{0,k,n,H}$ and~(\ref{eqlim-mh2q}) to obtain
\begin{eqnarray*}
\esp \Biggl[ \Biggl\llvert 2^{-n} \sum_{k=1}^n
\frac{\mathrm{e}^{2qw_{l_n}(k2^{-n})}} {
\mathrm{e}^{\psi(2q)} l_n^{-\psi(2q)}} \biggl(\frac{\tilde
\delta_{0,k,n,H}^{2q}} {\esp[\tilde\delta_{0,k,n,H}^{2q}]} - 1 \biggr) \Biggr
\rrvert^p \Biggr] & \leq &C l_n^{\varepsilon-\psi(2pq) + p\psi(2q)} \esp
\biggl[ \frac{\tilde\delta_{0,0,n,H}^{2pq}} {
(\esp[\tilde\delta_{0,0,n,H}^{2q}])^p} \biggr]
\\
& \leq& C l_n^{\varepsilon} 2^{n\{\psi(2pq)-p\psi(2q)\}} .
\end{eqnarray*}
Now, as in the proof of Proposition~\ref{propconv-as-SnMq}, we must choose
$\alpha$ small enough so that $2q<q_{\chi'}$, for $\chi'<\chi-\alpha$, and
$\varepsilon$ such that $2pq<q_{\chi'}$ with $p=1+\varepsilon$. Such a choice
and~(\ref{eqpsi-pq}) applied with $2pq<q_{\chi'}$ yield
\[
\esp \biggl[ \biggl\llvert \frac{\tilde S_{n,1}}{m_{n,H}(2q)} - \tilde T_{n,2q} \biggr
\rrvert^p \biggr]  \leq C 2^{-\varepsilon\chi} l_n^{\varepsilon}
2^{n\varepsilon(1+\chi
')} = C 2^{n \varepsilon(\chi'+\alpha-\chi)} .
\]
This last bound and Borel--Cantelli's lemma yield that
$m_{n,H}^{-1}(2q)\tilde
S_{n,1} - \tilde T_{n,2q} \to0$, a.s. This and~(\ref{eqa-atilde-ps})
finally prove
that $\tilde S_{n,1} \to m_H(2q)$ a.s.

In order to prove that $S_{n,2}\to0$ a.s., by stationarity and $2$-dependence
in $j$, it is enough to prove that, for $p=1+\varepsilon$,
%
\begin{eqnarray}
\esp \biggl[ \biggl\llvert 2^{n\tau_H(2q)} \sum_{0 \leq k\ne k' < 2^n}
\rho_{H,n}^r\bigl(0,0,k,k'\bigr)
a_{0,k,n,H}^{q} a_{0,k',n,H}^{q} \biggr
\rrvert^p \biggr] = \mathrm{O}\bigl(2^{(\varepsilon\chi-\eta) n}\bigr)
\label{eqappcota1789}
\end{eqnarray}
for some $\eta>0$ and apply Borel--Cantelli's lemma. Since all quantities
involved are nonnegative, we can use the bound~(\ref
{eqbound-cov-ludena}), and
thus it suffices to obtain a bound for
\[
\esp \biggl[ \biggl\llvert 2^{n\tau_H(2q)} \sum_{0 \leq k\ne k' < 2^n}
\bigl|k-k'\bigr|^{r(2H-2)} a_{0,k,n,H}^{q}
a_{0,k',n,H}^{q} \biggr\rrvert^p \biggr] .
\]
Define
\[
\tilde\delta_{k}^2  = \int_{\Delta_{n,k}}
\int_{\Delta_{n,k}} |u - v|^{2H-2} M_n(\mathrm{d}
u) M_n(\mathrm{d} v) .
\]
Then $\tilde a_{0,k,n,H} = \tilde\delta_{k} \mathrm{e}^{q
w_{l_n}(k2^{-n})}$ and
using the bound~(\ref{eqappcota162}) and~(\ref{eqpsi-pq}), we obtain
\[
\esp \biggl[ \biggl\llvert 2^{n\tau_H(2q)}\sum_{0 \leq k\ne k' < 2^n}
\bigl|k-k'\bigr|^{r(2H-2)} \bigl\{a_{0,k,n,H}^{q}
a_{0,k',n,H}^{q} - \tilde a_{0,k,n,H}^{q} \tilde
a_{0,k',n,H}^{q} \bigr\} \biggr\rrvert^p \biggr] = \mathrm{O}
\bigl(2^{(\varepsilon\chi-\eta) n}\bigr) .
\]
Thus, we need to obtain a bound for $\esp[S_{n,4}^p]$ where
\[
S_{n,4} = 2^{n\tau_H(2q)} \sum_{0 \leq k\ne k' < 2^n}
\bigl|k-k'\bigr|^{r(2H-2)} \tilde a_{0,k,n,H}^{q}
\tilde a_{0,k',n,H}^{q} ,
\]
which we further decompose as $S_{n,4}=S_{n,5}+S_{n,6}$ with
\begin{eqnarray*}
S_{n,5} & = &2^{n\tau_H(2q)} \sum_{0 \leq k\ne k' < 2^n}
\bigl|k-k'\bigr|^{r(2H-2)} \bigl\{\tilde\delta_k^q
\tilde\delta_{k'}^q -\esp\bigl[\tilde
\delta_k^q \tilde\delta_{k'}^q
\bigr]\bigr\} \mathrm {e}^{qw_{l_n}(k2^{-n})+qw_{l_n}(k'2^{-n})} ,
\\
S_{n,6} & =& 2^{n\tau_H(2q)} \sum_{0 \leq k\ne k' < 2^n}
\bigl|k-k'\bigr|^{r(2H-2)} \esp\bigl[\tilde\delta_k^q
\tilde\delta_{k'}^q\bigr] \mathrm {e}^{qw_{l_n}(k2^{-n})+qw_{l_n}(k'2^{-n})} .
\end{eqnarray*}
Since $H<3/4$ and $r\geq2$, we have that $r(2H-2)<-1$ and the series
$\sum k^{r(2H-2)}$ is summable. Thus, applying Cauchy--Schwarz' inequality yields
\begin{eqnarray*}
\esp \biggl[ \biggl| 2^{-n} \sum_{0 \leq k\ne k' < 2^n}
\bigl|k-k'\bigr|^{r(2H-2)} \mathrm{e}^{qw_{l_n}(k2^{-n})+qw_{l_n}(k'2^{-n})}
\biggr|^p \biggr] \leq C \esp \Biggl[ \Biggl\llvert 2^{-n}
\sum_{k=0}^{2^n-1} \mathrm{e}^{qw_{l_n}(k2^{-n})}
\Biggr\rrvert^{2p} \Biggr] .
\end{eqnarray*}
Next, applying Lemma~\ref{lemappmerida} with $p$ such that $2pq<q_\chi$
and $\varepsilon'<p-1$ yields
%
\begin{eqnarray}
\esp \biggl[ \biggl\llvert 2^{-n} \sum_{0 \leq k\ne k' < 2^n}
\bigl|k-k'\bigr|^{r(2H-2)} \mathrm{e}^{qw_{l_n}(k2^{-n})+qw_{l_n}(k'2^{-n})} \biggr
\rrvert^p \biggr] \leq C l_n^{-\{\psi(2pq)-\varepsilon'\}} .
\label{eqcota1515-corrige}
\end{eqnarray}
By (\ref{eqappcota1492}), it holds that $\esp[\tilde\delta_k^q\tilde
\delta_{k'}^q]
\sim C(k,k') l_n^{\psi(2q)} 2^{-n\zeta_H(q)}$ where $C(k,k')$ is uniformly
bounded, thus
\[
\esp\bigl[S_{n,6}^p\bigr] \leq C
l_n^{-\{\psi(2pq)-p\psi(2q)-\varepsilon'\}}.
\]
If $2pq<q_\chi$, applying~(\ref{eqpsi-pq}), we have
\begin{eqnarray*}
(1-\alpha) \bigl\{\psi(2pq)-p\psi(2q) -\varepsilon'\bigr\} -\varepsilon
\chi\leq (1-\alpha) \varepsilon(1+\chi) - \varepsilon' \leq \varepsilon-
\varepsilon' - \alpha\varepsilon(1+\chi),
\end{eqnarray*}
which can be made negative by choosing $\varepsilon'$ close enough to
$\varepsilon$.
To deal with the last term, as in the proof of
Proposition~\ref{propconv-as-SnMq} we use the conditional $2^{\alpha n}$
dependence of the random variables $\delta_k$. We obtain the bound
\[
\esp\bigl[S_{n,5}^p\bigr] \leq C 2^{n\{\psi(2pq)-p\psi(2q)-\varepsilon\}} = \mathrm{O}
\bigl(2^{n(\varepsilon\chi-\eta)}\bigr)
\]
for small some $\eta>0$. We have proved~(\ref{eqappcota1789}), and
thus~(\ref{eqratio}) holds.
We can now define
\[
\Gamma_1(q) = \sum_{r=2}^\infty
\frac{g_r(q)^2}{r!} \Gamma(r,q) .
\]
As $\sum_{r=2}^\infty(r!)^{-1} g_r(q)^2 <\infty$ and $\Gamma_n(r,q)\le
\Gamma_n(2,q)$, then by the bounded convergence theorem, the previous
series is
convergent and thus we have obtained that
\begin{eqnarray*}
L^{-1} 2^{n\tau_H(2q)} \operatorname{var}_M
\bigl(S_{L,n}(X,q)\bigr) \to\Gamma_1(q) \qquad \mbox{a.s.}
\end{eqnarray*}
This also yield that there exists a constant $\Gamma_2(q)$ such that
\[
L^{-1} 2^{n\tau_H(2q)} \operatorname{var}_M
\bigl(2^{\tau_H(q)}S_{L,n+1}(X,q)\bigr) \to \Gamma_2(q) \qquad
\mbox{a.s.}
\]
By similar techniques, we also obtain that there exists a constant
$\Gamma_3(q)$
such that
\[
L^{-1} 2^{n\tau_H(2q)} \operatorname{cov}_M
\bigl(S_{L,n}(X,q),S_{L,n+1}(X,q)\bigr) \to
\Gamma_3(q) \qquad\mbox{a.s.}
\]
Finally, since $T_n(X,q)= S_{L,n}(X,q) - 2^{\tau_H(q)} S_{L,n+1}(X,q)$,
the last
three convergences yield~(\ref{eqconv-var-TLnXq-H}).

\begin{pf*}{Proof of \eqref{eqconv-weak-M-H-2}}
By Nualart and Peccati \cite{nualartpecatti2005}, Theorem~1, the proof will follow by
checking that
%
\begin{eqnarray}
\label{eqconv-weak-M-H-1-p1} L^{-2} 2^{2n\tau_H(2q)}
\esp_{M}\bigl[ \bigl\{T_n(X,q) - \esp_M
\bigl[T_n(X,q)\bigr]\bigr\}^4\bigr] \to3
\Gamma(H,q)^2 \qquad\mbox{a.s. }
\end{eqnarray}
Define
\begin{eqnarray*}
&&T_{n,r} (X,q)
\\
&&\quad = \sum_{j=0}^{L-1} \sum
_{k=0}^{2^n-1} 2^{\tau_H(q)} \bigl
\{a_{j,2k,n+1,H}^q H_r(Y_{j,2k,n+1}) +
a_{j,2k+1,n+1,H}^q H_r(Y_{j,2k+1,n+1})\bigr\}\\
&&\hspace*{44pt}\quad{} -
a_{j,k,n,H}^q H_r(Y_{j,k,n}) .
\end{eqnarray*}
Then, from the definition of $T_n(X,q)$ and recalling the expansion
$G_q =
\sum_{r=2}^\infty\frac{g_r(q)}{r!} H_r$ in terms of the Hermite
polynomials, to
show (\ref{eqconv-weak-M-H-1-p1}) it is enough to check that
%
\begin{eqnarray}
\label{eqconv-weak-M-H-1-p2} \esp_M \bigl[\bigl(T_{n,r}(X,q)
\bigr)^4\bigr] = \frac{3}{(r!)^2} \esp_M^2
\bigl[T_{n,r}^2(X,q)\bigr] + R_n(q,r) ,
\end{eqnarray}
with $L^{-2} 2^{2n\tau_H(2q)}R_n(q,r) \to0$ a.s.
In order to calculate the fourth order moment in~(\ref{eqconv-weak-M-H-1-p2})
we use a standard application of the Diagram formula, for which we use the
notation in Surgailis \cite{surgailis}. Given a centered stationary Gaussian process
$\{X_j\}_{j\ge1}$ with positive covariance
$c(t_i,t_j)=\operatorname{cov}(X_{t_i},X_{t_j})$ and variance one, and a triangular
array of positive elements $\{b_t\}_{t=1}^N$ define $S_N(b):=\sum_{t=1}^N b_t
H_r(X_t)$. We introduce the following basic lattice notation. Let $W$
be a 4~row
table, whose rows correspond to the size $r$ vectors $W_i=(i,\ldots,i) ,
i=1,\ldots, 4$. Consider the collection $\Gamma$ of Gaussian flat connected
diagrams $\gamma$, that is, of partitions of $W$ defined by the
disjoint subsets
$\{V_\ell\}$ with $W=\bigcup_\ell V_\ell$, such that, respectively,
$|V_\ell|=2$,
no $V_\ell\subset W_i$ and it is not possible to write $W=W_1\cup W_2$, where
$W_1$ and $W_2$ can be partitioned by the diagram separately. Then, we have
that (see, e.g., Surgailis \cite{surgailis})
%
\begin{eqnarray}
\label{eqconv-weak-M-H-1-p3} \esp\bigl[\bigl(S_N(b)
\bigr)^4\bigr] &=& 3 \Biggl( \sum_{t_1,t_2=1}^N
b_{t_1} b_{t_2} c^{r}(t_1,t_2)
\Biggr)^2
\nonumber
\\[-8pt]
\\[-8pt]
\nonumber
&&{}+ \sum_{\gamma\in\Gamma} \sum
_{t_1,\ldots,t_4} b_{t_1}\cdots b_{t_4} \prod
_{1 \le i<j \le4} c^{l_{i,j}}(t_i,t_j)
,\qquad
\end{eqnarray}
where $l_{i,j}$ is the number of elements $V_\ell$ in the diagram that
pair row
$i$ with row $j$. Because the diagram is connected and each row must
appear at
least once, for each pair $i,j$ we have $1\le l_{i,j}<r$. Also, the
fact that
the diagrams in $\Gamma$ are flat (i.e., that no $V_\ell\subset W_i$) assures
that the second sum is over 4-tuples of pairwise distinct indices. On
the other
hand, since $0 \leq c(i,j) \leq1$ and $r\ge2$, for each $\gamma\in
\Gamma$,
by symmetry
%
\begin{eqnarray}\label{eqby-symmetry}
&&\sum_{t_1,\ldots,t_4} b_{t_1}\cdots
b_{t_4}\prod_{1\le i<j\le4} c^{l_{i,j}}(t_i,t_j)
\nonumber
\\[-8pt]
\\[-8pt]
\nonumber
&&\quad
\le\sum_{t_1,\ldots,t_4} b_{t_1}\cdots
b_{t_4} c(t_1,t_2)c(t_2,t_3)
c(t_3,t_4) c(t_4,t_1) .
\end{eqnarray}
Applying (\ref{eqconv-weak-M-H-1-p3}) and~(\ref{eqby-symmetry}) to
$T_{n,r}(X,q)$, we obtain~(\ref{eqconv-weak-M-H-1-p2}) if we show that
%
\begin{eqnarray}
\label{eqconv-weak-M-H-1-p6} && L^{-2} 2^{2n\tau_H(2q)} \sum
_{j_1,\ldots,j_4=1}^{L-1} \sum_{k_1,\ldots,k_4=1}^{2^n-1}
\prod_{1\le i\le4} a_{j_i,k_i,n,H}^q
\rho_{H,n}(j_1,j_2,k_1,k_2)
\rho_{H,n}(j_2,j_3,k_2,k_3)\qquad
\nonumber\\
&&\hphantom{L^{-2} 2^{2n\tau_H(2q)} \sum
_{j_1,\ldots,j_4=1}^{L-1} \sum_{k_1,\ldots,k_4=1}^{2^n-1}
\prod_{1\le i\le4}}{}\times\rho_{H,n}(j_3,j_4,k_3,k_4)\\
&&\hphantom{L^{-2} 2^{2n\tau_H(2q)} \sum
_{j_1,\ldots,j_4=1}^{L-1} \sum_{k_1,\ldots,k_4=1}^{2^n-1}
\prod_{1\le i\le4}}{}\times\rho_{H,n}(j_1,j_2,k_1,k_4)
\to0 \qquad \mbox{a.s.}\nonumber
\end{eqnarray}
The fact that the sum is over pairwise distinct indices assures that
$(j_i,k_i)\ne(j_\ell,k_\ell)$ for $i\ne\ell$, however it is necessary to
distinguish several cases:
\begin{itemize}
\item Case $j_i\equiv j$ for all $i=1,\ldots, 4$. We prove that
%
\begin{eqnarray}\label{eqconv-weak-M-H-1-p7}
&&L^{-2} 2^{2n\tau_H(2q)}\sum_{j=1}^{L-1}
\sum_{k_1,\ldots,k_4=1}^{2^n-1} \prod
_{1\le i\le4} a_{j,n,k_i,H}^q \rho_{H,n}(j_1,j_2,k_1,k_2)
\nonumber
\\
&&\hphantom{L^{-2} 2^{2n\tau_H(2q)}\sum_{j=1}^{L-1}
\sum_{k_1,\ldots,k_4=1}^{2^n-1} \prod
_{1\le i\le4}}{}\times \rho_{H,n}(j,j,k_2,k_3)
\rho_{H,n}(j,j,k_3,k_4)\\
&&\hphantom{L^{-2} 2^{2n\tau_H(2q)}\sum_{j=1}^{L-1}
\sum_{k_1,\ldots,k_4=1}^{2^n-1} \prod
_{1\le i\le4}}{}\times\rho_{H,n}(j,j,k_1,k_4)
\to0 \qquad\mbox{a.s.}\nonumber
\end{eqnarray}
This will be achieved by showing that the expectation of the l.h.s. of
(\ref{eqconv-weak-M-H-1-p7}) tends to zero. By stationarity of increments
and H\"older's inequality, we have
\begin{eqnarray*}
\esp \biggl[ \prod_{1\le i\le4} a_{0,n,k_i,H}^q
\biggr] \leq \esp^{1/2} \bigl[ a_{0,0,n,H}^{2q}
a_{0,k_2-k_1+1,n,H}^{2q} \bigr] \esp^{1/2} \bigl[
a_{0,0,n,H}^{2q} a_{0,k_4-k_3+1,n,H}^{2q} \bigr] .
\end{eqnarray*}
In addition, by the scaling property, we have that
\begin{eqnarray*}
&&\esp \bigl[ a_{0,0,n,H}^{2q} a_{0,k_2-k_1+1,n,H}^{2q}
\bigr] \\
&&\quad= 2^{-n\zeta_H(4q)} (k_2-k_1+1)^{\zeta_H(4q)-2\zeta_H(2q)}
C(k_1,k_2) ,
\end{eqnarray*}
with $C(k_1,k_2) \leq m_H(4q)$. This and the deterministic bounds on the
covariance (\ref{eqbound-cov-ludena}) yield that the expectation of the
l.h.s. of (\ref{eqconv-weak-M-H-1-p7}) is bounded by
\begin{eqnarray*}
&& L^{-1} 2^{-2n} 2^{n\{\psi(4q)-2\psi(q)\}} \sum
_{k_1,\ldots,k_4=0}^{2^n-1} |k_1-k_2|^{2H-2-(\psi(4q)-2\psi(2q))/2}
\\
&&\hspace*{114pt}\qquad{}  \times|k_3-k_4|^{2H-2-(\psi(4q)-2\psi(2q))/2}
|k_2-k_3|^{2H-2} \\
&&\hspace*{114pt}\qquad{}  \times|k_1-k_4|^{2H-2}
\\
&&\quad \leq C L^{-1} 2^{-2n} 2^{n\{\psi(4q)-2\psi(q)\}} \sum
_{k_1,k_2=0}^{2^n-1} |k_1-k_2|^{2(2H-2)}
\Biggl( \sum_{k=0}^{2^n-1} k^{2H-2-(\psi(4q)-2\psi(2q))/2}
\Biggr)^2
\\
&&\quad \leq C L^{-1} 2^{-n} \sum_{k}^{2^n-1}
k^{2(2H-2)} \Biggl( 2^{n\{\psi(4q)-2\psi(q)\}/2}\sum_{k=0}^{2^n-1}
k^{2H-2-(\psi(4q)-2\psi(2q))/2} \Biggr)^2 .
\end{eqnarray*}
Since $H>3/4$, the first series is summable, and since $\psi(4q)-2\psi(2q)>0$,
the second one is of order $n 2^{n(\{\psi(4q)-2\psi(2q)\}\vee
(4H-2))/2}$ (where
the factor $n$ only arises if the two exponents are equal). Recalling that
$\psi(4q)-2\psi(q)<1=\chi$ yields~(\ref{eqconv-weak-M-H-1-p7}).
\item Case $j_1=j_2=j_3=j$. In this case $|k_i-k_4|=\mathrm{O}(2^{-n(2H-2)})$, $i=1,2,3$
and by H\"older's inequality and independence of $a_{j',k_4,n,H}$ and
$\prod_{1\le i\le3} a_{j,k_i,n,H}$ we
have
\[
\esp \biggl[ a_{j',k_4,n,H}^q \prod_{1\le i\le3}
a_{j,k_i,n,H}^q \biggr] = \mathrm{O}\bigl(2^{-n\zeta(4q)/2}
2^{-n\zeta(2q)/2}2^{-n\zeta(q)}\bigr) |k_2-k_3|^{(\psi(4q)-2\psi(2q))/2}
.
\]
Using again the bound~(\ref{eqbound-cov-ludena}), we obtain
%
\begin{eqnarray}\label{eqconv-weak-M-H-1-p9}
&&
L^{-2}  2^{2n\tau_H(2q)} \sum
_{j=0}^{L-1} \sum_{j'=0}^{L-1}
\sum_{k_1,\ldots,k_4=1}^{2^n-1}
\esp \biggl[ \prod_{1\le i\le3} a_{j,k_i,n,H}^q
a_{j',k_4,n,H}^q \rho_{H,n}^2
\bigl(j,j',k_1,k_4\bigr)
\nonumber\\
&&\hspace*{139pt}\qquad{}\times\rho_{H,n}(j,j,k_2,k_3) \rho_{H,n}(j,j,k_3,k_1)
\biggr]\qquad
\\
 &&\quad = \mathrm{O}\bigl(L^{-1}2^{n(4H-3)}2^{-n(\psi(2q)/2-\psi(q)) }
\bigr) .\nonumber
\end{eqnarray}
As before, $2^{n(4H-3)}\to0$ under $H<3/4$ and $\psi(2q)/2-\psi(q)>0$ by
convexity of function~$\psi$.
\item Case $j_1=j_2$ and $j_3=j_4$. The bound for the expectation of the
l.h.s. of~(\ref{eqconv-weak-M-H-1-p6}) is then
%
\begin{eqnarray}
\label{eqconv-weak-M-H-1-p10}
&& L^{-2} 2^{2n\tau_H(2q)} \sum
_{j,j,j',j'}^{L-1} \sum_{k_1,\ldots,k_4=1}^{2^n-1}
\esp\bigl[ a_{j,k_1,n,H}^q a_{j,k_2,n,H}^q
a_{j',k_3,n,H}^q a_{j',k_4,n,H}^q
\nonumber\\
&&\hspace*{119pt}\quad{}\times\rho_{H,n}^2\bigl(j,j',k_1,k_4
\bigr)\nonumber
\\[-8pt]
\\[-8pt]
\nonumber
&&\hspace*{119pt}\quad{}\times \rho_{H,n}(j,j,k_1,k_2)
\rho_{H,n}\bigl(j',j',k_3,k_4
\bigr) \bigr] \\
&&\quad\le C 2^{n(4H-3)} ,\nonumber
\end{eqnarray}
by independence of $a_{j,n,k_1,H}^q$ and $a_{j',n,k_2,H}^q$ whenever
$j\ne j'$.\vspace*{1pt}
\item Case all $j_i$ are different. The bound is then
%
\begin{eqnarray}
\label{eqconv-weak-M-H-1-p11}
&&L^{-2} 2^{2n\tau_H(2q)} \sum
_{j_1,j_2,j_3,j_4}^{L-1} \sum_{k_1,\ldots,k_4=1}^{2^n-1}
\esp \biggl[ \prod_{1\le i\le4} a_{j_i,n,k_i,H}^q
\rho_{H,n}^2(j_1,j_2,k_1,k_4)
\nonumber
\\
&&\hspace*{152pt}\quad{}\times \rho_{H,n}(j_2,j_3,k_2,k_3)
\rho_{H,n}(j_3,j_4,k_3,k_4)
\biggr]\qquad\quad \\
&&\quad\le C 2^{n(-2\psi(2q)+4\psi(q))} 2^{n(2+\chi)(4H-3)} .\nonumber
\end{eqnarray}
As before, $2^{n(2+\chi)(4H-3)}\to0$ under $H<3/4$ and we use
$\psi(2q)>2\psi(q)$.
\end{itemize}

The proof follows by gathering (\ref{eqconv-weak-M-H-1-p7}),
(\ref{eqconv-weak-M-H-1-p9}), (\ref{eqconv-weak-M-H-1-p10}) and
(\ref{eqconv-weak-M-H-1-p11}).
\end{pf*}

\begin{appendix}\label{app}

\section*{Appendix: Additional lemmas}

\textit{Bounds for infinitely divisible random measures}.
We now state some results using the properties of infinitely divisible random
measures. The infinitely divisible measure $P$ introduced in
Section~\ref{secmrm} can be decomposed as $P=P_0+P_1$ where $P_0$ and
$P_1$ are
independent and
\[
\esp\bigl[ \mathrm{e}^{qP_i(A)} \bigr] = \mathrm{e}^{\mu(A)\psi_i(q)} ,
\]
with
\begin{eqnarray*}
\psi_0(q) & =& \frac{\sigma^2}2 q^2 + m q + \int
_{-1}^\infty\bigl\{\mathrm{e}^{qx} - 1 - q
x \mathbf 1_{\{|x|\leq1\}}\bigr\} \nu(\mathrm{d} x) ,
\\
\psi_1(q) & = &\int_{-\infty}^{-1} \bigl\{
\mathrm{e}^{qx} - 1 \bigr\} \nu(\mathrm {d} x) .
\end{eqnarray*}
Note that by assumption, $\psi_0$ is infinitely differentiable on
$[0,\infty)$,
whereas $\psi_1$ is infinitely differentiable on $(0,\infty)$ only.
Then, for
$A$ such that $\mu(A)\leq1$, $q>0$ and $p\geq1$ such that $pq < q^*$,
it holds that\vspace*{1pt}
%
\setcounter{equation}{0}
\begin{eqnarray}
\esp\bigl[\bigl|P_0(A)\bigr|^p\bigr] & = &\mathrm{O}\bigl(\bigl[\mu(A)
\bigr]^{(p/2)\wedge1}\bigr) , \label {eqappbound-esp}
\\
\esp\bigl[\bigl|\mathrm{e}^{qP_0(A)} - 1 - qP_0(A)\bigr|^p
\bigr] & =& \mathrm{O}\bigl(\mu(A)\bigr) , \label {eqapppoisson}
\\
\esp\bigl[\bigl|\mathrm{e}^{qP_1(A)} - 1 \bigr|^p\bigr] & =& \mathrm{O}\bigl(
\mu(A)\bigr) . \label{eqappvariance-finie}
\end{eqnarray}
Indeed, since $0 \leq\mathrm{e}^x-1-x \leq x^2\mathrm{e}^{x_+} \leq
x^2(\mathrm
e^x +1)$, with $x_+=\max(x,0)$, we have\vspace*{1pt}
\begin{eqnarray*}
&&\esp \bigl[ \bigl|\mathrm{e}^{qP_0(A)} - 1 - qP_0(A)\bigr|^p
\bigr]\\
&&\qquad \leq C \esp \bigl[ P_0^{2p}(A)
\mathrm{e}^{pqP_0(A)} \bigr] + C \esp \bigl[P_0^{2p}(A)
\bigr] .
\end{eqnarray*}
Denote $L(s) = \esp[ \mathrm{e}^{sP(A)}] = \mathrm{e}^{\psi_0(s)} \mu
(A)$. The
function $L$ is infinitely differentiable on $[0,q^*)$ and $L^{(n)}(q) =
\mathrm{O}(\mu(A))$ for all $q\geq0$ and $n\geq1$. This yields~(\ref{eqappbound-esp})
by the Cauchy--Schwarz inequality. Let $n$ be an integer greater than
$p$. Then,
for $0 \leq q <q^*$, (\ref{eqappvariance-finie}) follows from the following
bound:\vspace*{1pt}
\begin{eqnarray*}
\esp \bigl[ P_0^{2p}(A) \mathrm{e}^{pqP_0(A)}
\bigr] &\leq&\esp \bigl[ P_0^{2}(A) \mathrm{e}^{pqP_0(A)}
\bigr]+ \esp \bigl[ P_0^{2n}(A) \mathrm{e}^{pqP_0(A)}
\bigr] \\
&=& L^{\prime\prime}(pq)+ L^{(2n)}(pq) .\vspace*{1pt}
\end{eqnarray*}
To prove~(\ref{eqapppoisson}), note that $P_1(A)$ is a coumpond Poisson
distribution with negative jumps, thus $P_1(A)<0$ for all $A$, and for all
$p\geq1$,\vspace*{1pt}
\[
\esp\bigl[\bigl|\mathrm{e}^{qP_1(A)} - 1 \bigr|^p\bigr]  \leq1 -
\mathrm{e}^{\psi_1(q)\mu
(A)} = \mathrm{O}\bigl(\mu(A)\bigr) .
\]
Further, write\vspace*{1pt}
%
\begin{eqnarray}\label{eqapptrick}
&&\mathrm{e}^{qP(A)} - 1 - qP_0(A)
\nonumber
\\[-8pt]
\\[-8pt]
\nonumber
&&\quad= \bigl\{
\mathrm{e}^{qP_1(A)} - 1\bigr\} \mathrm {e}^{qP_0(A)} +
\mathrm{e}^{qP_0(A)} - 1 - qP_0(A)  .
\end{eqnarray}
This decomposition, (\ref{eqapppoisson}), (\ref{eqappvariance-finie}) and
the independence of $P_0$ and $P_1$ yield, for $q>0$ and $p\geq1$,\vspace*{1pt}
%
\begin{equation}
\esp\bigl[\bigl|\mathrm{e}^{qP(A)} - 1 - qP_0(A)\bigr|^p
\bigr]  = \mathrm{O}\bigl(\mu(A)\bigr) . \label{eqappcombine}\vspace*{1pt}
\end{equation}
Since $P$, $P_0$ and $P_1$ are independently scattered, these
inequalities yield
martingale maximal inequalities. For $A$ such that $\mu(A)\leq1$, and for
$C_u$ an increasing sequence of measurable subsets of $A$, it holds that\vspace*{1pt}
%
\begin{eqnarray}
\label{eqappmax-P0} \esp\Bigl[ \sup_u \bigl|P_0(C_u)\bigr|^p
\Bigr] & =& \mathrm{O}\bigl( \mu^{(p/2)\vee1}(A) \bigr) ,\qquad p \geq1 ,
\\
\esp \Bigl[ \sup_u\bigl|\mathrm{e}^{qP(C_u)}-1\bigr|^p
\Bigr] & =& \mathrm{O}\bigl(\mu (A)^{(p/2)\vee1}\bigr) ,\qquad p \geq1 , \label{eqappmax-qgeq2}
\\
\esp \Bigl[ \sup_{u}\bigl|\mathrm{e}^{qP(C_u)}-1-qP_0(C_u)\bigr|^p
\Bigr] & = &\mathrm{O}\bigl(\mu(A)\bigr) , \qquad p \geq1 . \label{eqappideal-bound}
\end{eqnarray}

\textit{Approximation and covariance bounds for the MRM}.

\begin{lem}
\label{lemappmerida}
Let $\alpha=1/n_0$ for some arbitrary integer $n_0\ge2$. For all
$p>1$ such
that $\esp[\mathrm{e}^{pqw_l(0)}]<\infty$, for any $\varepsilon'\in
(0,p-1)$, there
exists a constant $C$ such that
%
\begin{eqnarray}
\label{eqappcota-integral} \esp \biggl[ \biggl( \int_0^1
\frac{\mathrm
e^{qw_{l_n}(u)}}{\esp[\mathrm{e}^{qw_{l_n}(0)}]} \,\mathrm{d} u \biggr)^p \biggr] &\leq &C
l_n^{-\{\psi(pq) -p\psi(q) - \varepsilon'\}},
\\
\esp \Biggl[ \Biggl( 2^{-n} \sum_{k=0}^{2^{n}-1}
\frac{\mathrm
e^{qw_{l_n}(k2^{-n})}}{\esp[\mathrm{e}^{qw_{l_n}(0)}]} \Biggr)^p \Biggr] &\leq& C l_n^{-\{\psi(pq) -p\psi(q) -
\varepsilon'\}}
. \label{eqappcota-sum}
\end{eqnarray}

\end{lem}

\begin{pf}
The choice of $\alpha$ implies that
$(1-\alpha)n_0=n_0-1$ is an integer. Denote $g_n(u) = \mathrm
e^{qw_{l_n}(u)}/\esp[\mathrm{e}^{qw_{l_n}(0)}]$. Fix some integer
$k_0$, and
define $n_1 = k_0n_0$. If $n_1<n$, then
%
\begin{eqnarray}
\label{eqappdecom-integrale} \int_0^1
g_n(u) \,\mathrm{d} u & = &\int_0^1
g_{n_1} (u) \,\mathrm{d} u + \int_0^1
\bigl\{g_n(u) - g_{n_1}(u) \bigr\} \,\mathrm{d} u
\nonumber
\\[-8pt]
\\[-8pt]
\nonumber
& = &\int_0^1 g_{n_1}(u) \,
\mathrm{d} u + \sum_{k=0}^{2^{(1-\alpha)n_1}-1} \int
_{\Delta_{k,(1-\alpha)n_1}} \bigl\{g_n(u) - g_{n_1}(u) \bigr
\} \,\mathrm{d} u .
\end{eqnarray}
We bound the first integral by applying Jensen's inequality:
%
\begin{equation}
\esp \biggl[ \biggl( \int_0^1
g_{n_1}(u) \,\mathrm{d} u \biggr)^p \biggr] \leq\esp
\bigl[g_{n_1}^p(0)\bigr] = 2^{(1-\alpha)n_1\{\psi(pq)-p\psi(q)\}
} .
\label{eqappborne-first-integral}
\end{equation}
Since $w_{l_{n_1}}$ is independent of $w_{l_{n}}-w_{l_{n_1}}$, we can write
\[
g_n(u) - g_{n_1}(u) = g_{n_1}(u) \biggl\{
\frac{\mathrm
e^{qw_{l_n}(u)-qw_{l_{n_1}}(u)}}{\esp[\mathrm
e^{qw_{l_n}(0)-qw_{l_{n_1}}(0)}]} - 1 \biggr\}.
\]
Thus we see that the integrals $\int_{\Delta_{j,n_1}}
\{g_n(u)-g_{n_1}(u)\} \,\mathrm{d} u$ are centered and 2-dependent
conditionally on $\mathcal{F}_{n_1}$ the sigma-field generated by
$\{w_{l_{n_1}}(u), u\in[0,1]\}$. Thus by
von Bahr and Esseen~\cite{vonbahresseen1965}, Theorem~2, there is a constant $C$ such that
\begin{eqnarray*}
&&\esp \Biggl[ \Biggl\llvert \sum_{k=0}^{2^{(1-\alpha)n_1}-1}
\int_{\Delta_{k,(1-\alpha)n_1}} \bigl\{g_n(u) - g_{n_1}(u)
\bigr\} \,\mathrm{d} u \Biggr\rrvert^p \Biggr]
\\
&&\quad\leq C 2^{(1-\alpha)n_1} \esp \biggl[ \biggl\llvert \int_{\Delta_{0,(1-\alpha)n_1}}
\bigl\{ g_n(u) - g_{n_1}(u) \bigr\} \,\mathrm{d} u \biggr
\rrvert^p \biggr]
\\
&&\quad \leq C 2^{p-1} 2^{(1-\alpha)n_1} \esp \biggl[ \biggl\llvert \int
_{\Delta_{0,(1-\alpha)n_1}} g_n(u) \,\mathrm{d} u \biggr
\rrvert^p \biggr] + C 2^{p-1} 2^{(1-\alpha)n_1} \esp \biggl[
\biggl\llvert \int_{\Delta
_{0,(1-\alpha)n_1}} g_{n_1}(u) \,\mathrm{d} u
\biggr\rrvert^p \biggr]
\\
&&\quad \leq C 2^{p-1} 2^{(1-\alpha)n_1} \esp \biggl[ \biggl( \int
_{\Delta
_{0,(1-\alpha)n_1}} g_n(u) \,\mathrm{d} u \biggr)^p
\biggr] + C 2^{p-1} 2^{\{1-p+\psi(pq)-p\psi(q)\}(1-\alpha)n_1} .
\end{eqnarray*}
Since $l_n/l_{n_1} = l_{n-n_1}$, by the scaling property~(\ref
{eqscaling}), we have
\begin{eqnarray*}
\int_{\Delta_{0,(1-\alpha)n_1}} \mathrm{e}^{qw_{l_n}(u)} \,\mathrm{d} u =
l_{n_1} \int_0^{1}
\mathrm{e}^{qw_{l_{n-n_1}l_{n_1}}(l_{n_1}u)} \,\mathrm{d} u \stackrel{\mathrm{law}}= l_{n_1}
\mathrm{e}^{q\Omega_{l_{n_1}}} \int_0^1 \mathrm
e^{qw_{l_{n-n_1}}(u)} \,\mathrm{d} u .
\end{eqnarray*}
Thus,
\begin{eqnarray*}
\esp \biggl[ \biggl( \int_{\Delta_{0,(1-\alpha)n_1}} g_n(u) \,
\mathrm{d} u \biggr)^p \biggr] & =& 2^{(1-\alpha)n_1(\psi(pq)-p)}
\frac{(\esp[\mathrm
e^{qw_{l_{n-n_1}}(0)}])^p}{(\esp[\mathrm{e}^{qw_{l_{n}}(0)}])^p} \esp \biggl[ \biggl( \int_0^1
g_{n-n_1}(u) \,\mathrm{d} u \biggr)^p \biggr]
\\
& =& 2^{(1-\alpha)n_1(\psi(pq)-p\psi(q)-p)} \esp \biggl[ \biggl( \int_0^1
g_{n-n_1}(u) \,\mathrm{d} u \biggr)^p \biggr] .
\end{eqnarray*}
Thus we have obtained
%
\begin{eqnarray}
\label{eqappborne-second-integral}
&& \esp \Biggl[ \Biggl\llvert \sum
_{k=0}^{2^{(1-\alpha)n_1}-1} \int_{\Delta_{k,(1-\alpha)n_1}} \bigl
\{g_n(u) - g_{n_1}(u) \bigr\} \,\mathrm{d} u \Biggr
\rrvert^p \Biggr]
\nonumber
\\[-8pt]
\\[-8pt]
\nonumber
&&\quad\leq C 2^{(1-\alpha)n_1(\psi(pq)-p\psi(q)-p)} \esp \biggl[ \biggl( \int_0^1
g_{n-n_1}(u) \,\mathrm{d} u \biggr)^p \biggr] .
\end{eqnarray}
Denote $u_n = \esp [ ( \int_0^1 g_n(u) \,\mathrm{d} u  )^p
]$. Gathering~(\ref{eqappdecom-integrale}),~(\ref
{eqappborne-first-integral})
and~(\ref{eqappborne-second-integral}), we obtain the following recurrence:
\[
u_n \leq B + C2^{(1-\alpha)n_1(1-p+\psi(pq)-p\psi(q))} u_{n-n_1} .
\]
By choosing $k_0$ large enough, this yields that for any
$\varepsilon'\in(0,\varepsilon)$,
\[
u_n \leq B + 2^{(1-\alpha)n_1(\psi(pq)-p\psi(q)-\varepsilon')} u_{n-n_1} .
\]
Thus, there exists a constant $D$ such that
\[
u_n \leq D 2^{(1-\alpha)n(\psi(pq)-p\psi(q)-\varepsilon')} .
\]
This proves~(\ref{eqappcota-integral}). The bound~(\ref{eqappcota-sum})
follows by replacing the measure $\mathrm{d} u$ with a discrete
measure.
\end{pf}

\begin{lem}
\label{lemappapprox-MnDeltan}
Let $0<\alpha<1$ and $l_n=2^{-(1-\alpha)n}$. For $p\ge1$ and $q> 0$
such that $pq<q_\chi$, there exists a positive constant
$C$ such that
%
\begin{eqnarray}
\lim_{n\to\infty} 2^{n\zeta(q)} \mathrm{e}^{\psi(q)}
l_n^{-\psi(q)} \esp\bigl[M_n^q(
\Delta_{0,n})\bigr] &=& m(q) , \label{eqappmndeltan}
\\
\esp \bigl[ \bigl\llvert \mathrm{e}^{qw_{l_n}(0)} M_n^q(
\Delta_{0,n}) - M^q(\Delta_{0,n}) \bigr
\rrvert^p \bigr] &\leq&  C 2^{-\alpha(q\vee1)n/2} 2^{-n\zeta(pq)} .
\label{eqappcota16}
\end{eqnarray}
\end{lem}
\begin{pf}
Note that~(\ref{eqappcota16}) implies~(\ref{eqappmndeltan}). So we
only need to
prove~(\ref{eqappcota16}). Define the sets $I_n$, $B_n(u)$, $u\in
[0,2^{-n}]$ by
\begin{eqnarray*}
I_n = \bigcap_{0\leq u \leq2^{-n}}
A_{l_n}(u) = A_{l_n(0)} \cap A_{l_n(2^{-n})} ,\qquad
B_n(u) = A_{l_n}(u) \setminus I_n .
\end{eqnarray*}
See Figure~\ref{figappIBC} for an illustration. By definition of the
function $\psi$ and the measure $\mu$, we have, $\esp[\mathrm
e^{qP(I_n)}] = \mathrm{e}^{\psi(q) \mu(I_n)}$ and
\begin{eqnarray*}
\mu(I_n) & = &\int_{I_n} \frac{\mathrm{d} s \,\mathrm{d} t}{t^2} =
\int_{l_n}^1 \frac{t-2^{-n}}{t^{2}} \,\mathrm{d} t +
\int_1^\infty \frac{1-2^{-n}}{t^{2}} \,\mathrm{d} t
\\
& =& -\log(l_n) -2^{-n}\bigl(l_n^{-1}
- 1\bigr) + 1 - 2^{-n} = 1 - \log(l_n) -2^{-\alpha n} =
\mu\bigl(A_{l_n}(0)\bigr) - 2^{-\alpha n} .
\end{eqnarray*}
This yields $w_{l_n}(u) = P(I_n) + P(B_n(u))$ where the two summands are
independent and
%
\begin{equation}
\esp\bigl[\mathrm{e}^{qP(I_n)}\bigr] = \esp\bigl[ \mathrm{e}^{qw_{l_n}(0)}
\bigr] \bigl\{1 + \mathrm{O}\bigl(2^{-\alpha n}\bigr) \bigr\} . \label{eqappwln-In}
\end{equation}
Write further
\begin{eqnarray*}
M(\Delta_{0,n}) & =& \int_0^{2^{-n}}
\mathrm{e}^{w_{l_n}(u)} M_n(\mathrm {d} u) = \mathrm{e}^{P(I_n)}
\int_0^{2^{-n}} \mathrm{e}^{P(B_n(u))}
M_n(\mathrm{d} u)
\\
& = &\xi_n \int_0^{2^{-n}}
\mathrm{e}^{P(B_n(u))} \bar M_n(\mathrm{d} u) ,
\end{eqnarray*}
with $\xi_n = \mathrm{e}^{P(I_n)} M_n(\Delta_{0,n}) $ and $\bar
M_n(\mathrm{d}
u) = M_n(\mathrm{d} u)/M_n(\Delta_{0,n})$ is a random probability measure
on~$\Delta_{0,n}$. We thus obtain
\begin{eqnarray*}
M^q(\Delta_{0,n}) - \mathrm{e}^{qw_{l_n}(0)}
M_n^q(\Delta_{0,n}) = \xi_n^q
\biggl\{ \biggl( \int_0^{2^{-n}}
\mathrm{e}^{P(B_n(u))} \bar M_n(\mathrm{d}u) \biggr)^q
- \mathrm{e}^{qP(B_n(0))} \biggr\} .
\end{eqnarray*}
Noting that for $x >-1$ and $q>0$, it holds that $0 \leq|1-(1+x)^q|
\leq C_q
(|x|+|x|^q)$ and since $P(I_n)$, $M_n(\Delta_{0,n})$ and $P(B_n(u))$,
$0\leq u
\leq2^{-n}$, are mutually independent, we have
\begin{eqnarray*}
&&\esp \biggl[ \xi_n^{pq} \biggl\llvert \biggl(\int
_0^{2^{-n}} \mathrm{e}^{P(B_n(u))}
\bar{M}_n(\mathrm{d} u) \biggr)^q - \mathrm{e}^{qP(B_n(0))}
\biggr\rrvert^p \biggr]
\\
&&\quad\leq C \esp\bigl[\xi_n^{pq}\bigr] \Bigl\{ \Bigl(\esp
\Bigl[ \sup_{0\leq u \leq2^{-n}} \bigl\llvert \mathrm{e}^{P(B_n(u))} - 1 \bigr
\rrvert^{p(q\vee1)} \Bigr] \Bigr)^{q\wedge1} + \esp \Bigl[
\sup_{0\leq u \leq2^{-n}} \bigl\llvert \mathrm e^{P(B_n(u))} - 1 \bigr
\rrvert^{p} \Bigr] \Bigr\} .
\end{eqnarray*}
Thus, applying~(\ref{eqappmax-qgeq2}) yields
\[
\esp \bigl[\bigl\llvert M^q(\Delta_{0,n}) -
\mathrm{e}^{qw_{l_n}(0)} M_n^q(\Delta_{0,n})
\bigr\rrvert^p \bigr]  = \mathrm{O}\bigl(2^{-\alpha n(q\wedge1)/2}\bigr) \esp\bigl[
\xi_n^{pq}\bigr] .
\]
\upqed\end{pf}

\renewcommand{\thefigure}{\arabic{figure}}
\setcounter{figure}{1}
\begin{figure}[t]\vspace*{21pt}
\centering\setlength{\unitlength}{1mm}
\begin{picture}(100,80)(-55,-48)

\multiput(-60,-20)(5,0){20}{\line(1,0){3}}
\put(-60,-40){\line(1,0){100}}
\multiput(-60,5)(5,0){20}{{\line(1,0){3}}}
\put(45,-22){$l_n$}
\put(45,5){$1$}
\put(45,-41){$0$}

\put(-40,15){$B_n(u)$}
\put(-10,15){$I_n$}
\put(10,15){$B_n(u)$}

\put(-21,-44){$ 0$}
\put(-11,-44){$u$}
\put(-2,-45){${2^{-n}}$}

\put(-54,35){${-\frac12}$}
\put(-45,35){$u-\frac12$}
\put(-34,35){${\frac1{2^n}-\frac12}$}
\put(9,35){${\frac12}$}
\put(15,35){$\frac12+u$}
\put(26,35){${\frac12+\frac1{2^n}}$}

\multiput(-20,-41)(0,5){6}{{\line(0,1){3}}}
\multiput(-10,-41)(0,5){6}{{\line(0,1){3}}}
\multiput(0,-41)(0,5){6}{{\line(0,1){3}}}

\thicklines
\put(-33,-20){{\line(1,0){25}}}
\put(-50,5){{\line(0,1){25}}}
\put(-33,-20){{\line(-2,3){17}}}
\put(9,5){{\line(0,1){25}}}
\put(-8,-20){{\line(2,3){17}}}

\put(-23,-20){\line(1,0){25}}
\put(-40,5){\line(0,1){25}}
\put(-23,-20){\line(-2,3){17}}
\put(20,5){\line(0,1){25}}
\put(3,-20){\line(2,3){17}}

\put(-13,-20){{\line(1,0){26}}}
\put(-30,5){{\line(0,1){25}}}
\put(-13,-20){{\line(-2,3){17}}}
\put(30,5){{\line(0,1){25}}}
\put(13,-20){{\line(2,3){17}}}

\end{picture}
\caption{The sets $I_n$ and $B_n(u)$.}
\label{figappIBC}
\end{figure}
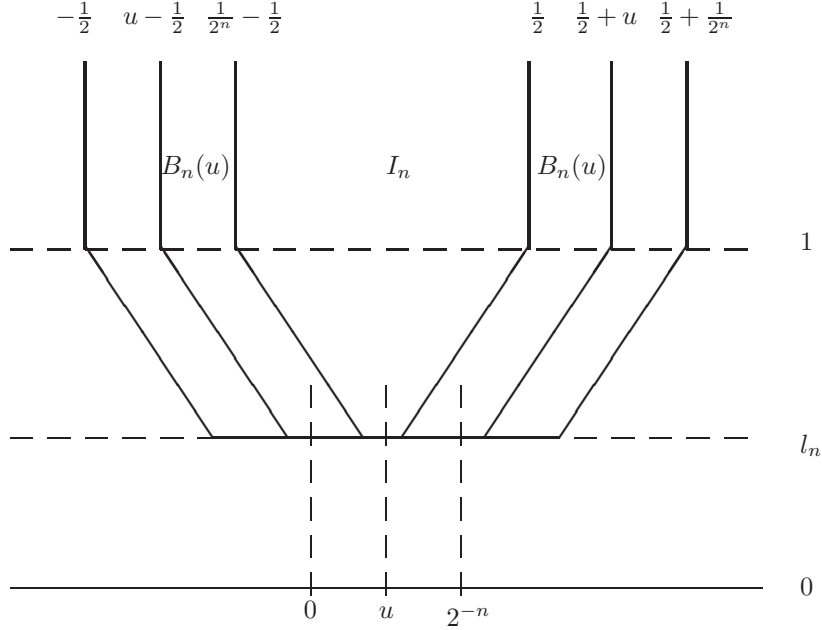

\begin{figure}[t] \centering\setlength{\unitlength} {0.9mm} %
\vspace*{-20pt}
\begin{picture}(115,95) (-50,-45) 
\put(-65,-40){\line(1,0){125}}
\multiput(-65,-18) (5,0){25} {\line(1,0){3}} \multiput(-65,5) (5,0){25} {
\line(1,0){3}} \multiput(-30,-41) (0,5){6} {{\line(0,1){3}}} \multiput(-20,-41)
(0,5){6} {{\line(0,1){3}}} \multiput(10,-41) (0,5){6} {{\line(0,1){3}}}
\multiput(20,-41) (0,5){6} {{\line(0,1){3}}} \put(62,4){$1$} \put(62,-19){$l=1-s-t$}
\put(62,-42){$0$} \put(-21,-45){$ s$} \put(-54,37){${s-\frac12}$} \put(5,37){${
\frac12+s}$} \put(-5,37){${\frac12+u}$} \put(-31,-45){$ u$} \put(-64,37){${u-
\frac12}$} \put(20,-45){$ v$} \put(-15,37){${v-\frac12}$} \put(47,37){${v+
\frac12}$} \put(-40,22){$A_{s,t}$} \put(20,22){$B_{s,t}$}
\put(-8,22){$C_{u,v}$} \put(1,22){$D_{s,u}$}
\put(-59,22){$D'_{s,u}$} \put(-19,22){$E_{t,v}$}
\put(41,22){$E'_{t,v}$} \put(5,-45){${1-t}$}
\put(-25,37){${\frac12-t}$} \put(35,37){${\frac32-t}$} \thicklines \put(-45,-18){{
\line(1,0){40}}} \put(-5,-18){{\line(1,0){40}}} \put(-50,5){{\line(0,1){30}}}
\put(-34.6,-18){{\line(-2,3){15.4}}} \put(10,5){{\line(0,1){30}}} \put(-5.5,-18){{
\line(2,3){15.4}}} \put(-60,5){{\line(0,1){30}}} \put(0,5){{\line(0,1){30}}}
\put(-15.3,-18){{\line(2,3){15.4}}} \put(-44.6,-18){{\line(-2,3){15.4}}}
\put(-20,5){{\line(0,1){30}}} \put(-4.5,-18){{\line(-2,3){15.4}}} \put(40,5){{
\line(0,1){30}}} \put(24.5,-18){{\line(2,3){15.4}}} \put(-10,5){{\line(0,1){30}}}
\put(5.5,-18){{\line(-2,3){15.4}}} \put(50,5){{\line(0,1){30}}} \put(34.5,-18){{
\line(2,3){15.4}}} \end{picture}
\caption{The sets
$A,B,C,D,D',E,E'$.} \label{figappABCDE} \end{figure}
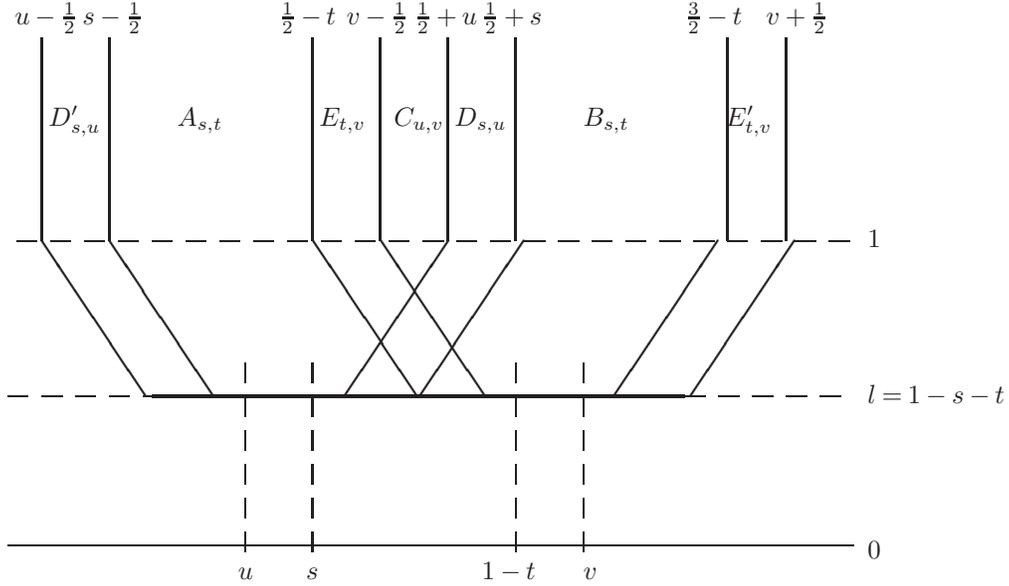

\begin{lem}
\label{lemappeven-better}
If $q+q' < q_{\max}$, then for $s,t\in(0,1)$ such that $s + t < 1/2$,
%
\begin{eqnarray}
\label{eqappcov-plus-fort} \cov\bigl(M^q\bigl([0,s]\bigr),M^{q'}
\bigl([1-t,1]\bigr)\bigr) = \mathrm{O} \bigl( (s+t)^{\{\zeta(q)+\zeta(q')+1\} } \bigr) .
\end{eqnarray} %
\end{lem}

\begin{pf}
Define $l=1-s-t$ and
$M_l(\mathrm{d} u) = \mathrm{e}^{-w_l(u)} M(\mathrm{d} u)$.
By construction, the measure $M_l$ is independent of $
\{w_l(u) \}$ and $M_l ([0,s] )$ is
independent of $M_l ([1-t,1] )$. Define the sets
$A_{s,t}$ and $B_{s,t}$ by %
\[
A_{s,t} = A_l(s) \setminus A_l(1-t) ,\qquad
B_{s,t} = A_l(1-t) \setminus A_l(s).
\]
For $u \leq s$ and $v \geq1-t$, define %
\begin{eqnarray*}
C_{u,v} &=& A_l(u) \cap A_l(v) ,
\\
D_{s,u} &=& C_{s,v} \setminus C_{u,v} ,\qquad
D'_{s,u} = A_l(u) \setminus
A_l(s),
\\
E_{t,u} &=& C_{u,1-t} \setminus C_{u,v} ,\qquad
E'_{t,v} = A_l(v) \setminus
A_l(1-t) .
\end{eqnarray*} %
See Figure~\ref{figappABCDE}
for an illustration. Note that all these sets are above the horizontal line at
level $l=1-s-t$, hence $P(A)$ is independent of $M_l$ and $P(A)$ is
independent of $P(B)$, where $A,B$ are any two of these sets. Note also that $
\bigcup_{u \leq s, v \geq1-t} C_{u,v} = C_{s,1-t}$,\vspace*{1pt} $D_{s,u}
\subset C_{s,1-t}$, $E_{t,v} \subset C_{s,1-t}$,
$D'_{s,u} \subset D'_{s,0}$ and
$E'_{t,v}\subset E'_{t,1}$. We
moreover have %
\begin{eqnarray*}
\mu(A_{s,t}) &= &\mu(B_{s,t}) = 1 ,
\\
\mu(C_{s,1-t}) &=& -\log(1-s-t) ,
\\
\mu \bigl(D'_{s,0} \bigr) &= &\frac s{1-s-t} ,\qquad \mu
\bigl(E'_{t,1} \bigr) = \frac t{1-s-t} .
\end{eqnarray*} %
Moreover, for $u\leq s$ and $v\geq1-t$, we have the
following decompositions: %
\begin{eqnarray*}
w_l(u) & =& P(A_{s,t}) + P(C_{u,v}) + P
\bigl(D'_{s,u} \bigr) + P(E_{t,v}) ,
\\
w_l(v) & =& P(B_{s,t}) + P(C_{u,v}) +
P(D_{s,u}) + P \bigl(E'_{t,v} \bigr) .
\end{eqnarray*} %
Recall that the random measure $P$ can be split into
two independent random measures $P_0$ and $P_1$ such
that $P=P_0+P_1$. For $i=0,1$ and $u\in[0,s]$, define
$\pi_{i,l}(u) = P_i (D'_{s,u}
)+P_i(C_{u,1-t})$ and %
\[
\pi_l(u) = \pi_{0,l}(u) + \pi_{1,l}(u) .
\]
Similarly, for $i=0,1$ and $v\in[1-t,1]$, define $\pi'_{i,l}(v) =
P_i(E'_{t,v})+P_i(C_{s,v})$ and
\[
\pi'_l(v) = \pi'_{0,l}(v) +
\pi'_{1,l}(v) .
\]
Let $\bar M_l$ and $\bar M'_l$ denote the normalized measures $M_l/M_l([0,s])$
and $M_l/M_l([1-t,1])$ and
\begin{eqnarray*}
\zeta_l&=& M_l \bigl([0,s] \bigr) ,\qquad
\xi_l = M_l \bigl([1-t,1] \bigr) ,
\\
\gamma_l &=& \int_0^{s} \bigl\{
\mathrm{e}^{\pi_{l}(u)}-1 \bigr\} \bar M_l(\mathrm {d} u) ,\qquad
\gamma'_l = \int_{1-t}^1
\bigl\{\mathrm{e}^{\pi'_l(v)}-1 \bigr\} \bar M'_l(
\mathrm {d} v) ,
\\
R_l& =& (1+\gamma_l)^q- 1 - q
\gamma_l ,\qquad R'_l \bigl(1+
\gamma'_l \bigr)^{q'}- 1 - q'
\gamma'_l .
\end{eqnarray*}
This yields
\begin{eqnarray*}
M^q \bigl([0,s] \bigr) & = &\mathrm{e}^{qP(A_{s,t})}
\zeta_l^q \times \{ 1 + q\gamma_l +
R_l \} ,
\\
M^{q'} \bigl([1-t,1] \bigr) & =& \mathrm{e}^{q'P(B_{s,t})}
\xi_l^{q'} \times \bigl\{ 1 + q'
\gamma_l+R'_l \bigr\} .
\end{eqnarray*}
Note that $\zeta_l$ and $\xi_l$ are independent and independent of $\pi_l$ and
$\pi'_l$ which are independent of $M_l$. Thus, $\xi_l$ is also
independent of
$\gamma_l$ and $R_l$, and $\zeta_l$ is independent of $\gamma'_l$ and
$R'_l$. Also, $P(A_{s,t})$ and $P(B_{s,t})$ are independent of all the other
quantities, and $\esp[\mathrm{e}^{qP(A_{s,t})}] = \esp[\mathrm
{e}^{qP(B_{s,t})}] =
\mathrm{e}^{\psi(q)}$. Thus,
%
\begin{eqnarray}
&&\mathrm{e}^{-\psi(q)-\psi(q')}  \cov \bigl(M^q \bigl([0,s]
\bigr),M^{q'} \bigl([1-t,1] \bigr) \bigr)
\nonumber
\\
&&\quad = qq'\cov \bigl(\zeta_l^q
\gamma_l, \xi_l^{q'}\gamma'_l
\bigr) + q \esp \bigl[\xi_l^q\zeta_l^{q'}
\gamma_lR'_l \bigr] - q \esp \bigl[
\zeta_l^q\gamma_l \bigr] \esp \bigl[
\xi_l^{q'}\gamma'_l \bigr]
\label{eqappcov-1}
\\
&&\qquad{} + q' \esp \bigl[\zeta_l^q
\xi_l^{q'} R_l \gamma'_l
\bigr] - q' \esp \bigl[\zeta_l^q
R_l \bigr] \esp \bigl[\xi_l^{q'}
\gamma'_l \bigr] + \esp \bigl[\zeta_l^q
\xi_l^{q'} R_l R'_l
\bigr] - \esp \bigl[\zeta_l^q R_l \bigr]
\esp \bigl[\xi_l^{q'}R'_l
\bigr] . \label{eqappcov-2}\qquad\quad
\end{eqnarray}
We will show that all the terms on the right-hand side are of order $(s+t)^{-1}
\esp[\xi_k^q]\esp[\zeta_k^{q'}]$. Since $\pi_{l}$ and $\pi'_l$ are independent
of the measure $M_l$, using the definition of $\pi_l$ and $\pi'_l$ and
the fact
that the random measure $P'$ has independent increments, and $\esp
[\mathrm
e^{P(A)}] = 1$ for all measurable set $A$ with finite $\mu$ measure, we have
\begin{eqnarray*}
\cov \bigl(\zeta_l^q\gamma_l,
\xi_l^{q'}\gamma_l \bigr) & = &\esp \biggl[
\zeta_l^q\xi_l^{q'} \int
_0^{s}\int_{1-t}^1
\cov \bigl(\mathrm{e}^{\pi_l(u)},\mathrm{e}^{\pi'_{l}(v)} \bigr) \bar
M_l(\mathrm{d} u) \bar M'_l(\mathrm{d}
v) \biggr]
\\
& = &\esp \biggl[ \zeta_l^q \xi_l^{q'}
\int_0^{s}\int_{t}^1
\operatorname{var} \bigl(\mathrm{e}^{P(C_{u,v})} \bigr) \bar
M_l( \mathrm{d} u) \bar M'_l(\mathrm{d}
v) \biggr]
\\
& = &\esp \biggl[ \zeta_l^q \xi_l^{q'}
\int_0^{s}\int_{1-t}^1
\bigl\{\mathrm{e}^{\psi(2)\mu(C_{u,v})}-1 \bigr\} \bar M_l(\mathrm{d} u)
\bar M'_l(\mathrm{d} v) \biggr]
\\
& \leq&\esp \bigl[ \zeta_l^q \bigr] \esp \bigl[
\xi_l^{q'} \bigr] \bigl\{\mathrm e^{\psi(2)\mu(C_{s,1-t})}-1
\bigr\} \leq C \esp \bigl[ \zeta_l^q \bigr] \esp \bigl[
\xi_l^{q'} \bigr] (s+t) .
\end{eqnarray*}
If $q>0$, a second order Taylor expansion yields that there exists a constant
$C_q\geq1$ such that for all $x\geq-1$,
%
\begin{equation}
\bigl|(1+x)^q - 1 - q x\bigr| \leq C_q \bigl(x^2+
|x|^{q\vee2} \bigr) . \label{eqapptrivial}
\end{equation}
Applying~(\ref{eqapptrivial}) and Jensen's inequality (since by
definition $\bar
M_l$ is a probability measure on $[0,s]$), we obtain, with $r_{l} =
\sup_{u\in[0,s]} |\mathrm{e}^{\pi_l(u)} - 1 |$, which is independent of $M_l$,
\begin{eqnarray*}
|R_l|  \leq C \int_0^s \bigl
\{ \bigl\llvert \mathrm{e}^{\pi_l(u)} - 1 \bigr\rrvert^{q\vee2} +
\bigl\llvert \mathrm{e}^{\pi_l(u)} - 1 \bigr\rrvert^{2} \bigr\}
\bar M_l(\mathrm{d}u) \leq C \bigl(r_{l}^{q\vee2}+r_{l}^2
\bigr).
\end{eqnarray*}
Define $r'_l = \sup_{u\in[0,s]} |\mathrm e^{\pi_l(u)} - 1 |$ and note that
$|\gamma_l| \leq r_l$ and $|\gamma'_l| \leq r'_l$. We thus get
\begin{eqnarray*}
\esp \bigl[\zeta_l^q\xi_l^{q'}
R_l \gamma'_l \bigr] \leq\esp \bigl[
\zeta_l^q\xi_l^{q'} \bigr] \esp
\bigl[ \bigl(r_l^2+r_i^{q\vee2}
\bigr) r'_l \bigr] \leq\esp \bigl[
\zeta_l^q \xi_l^{q'} \bigr]
\esp^{1/2} \bigl[ \bigl(r_l^2+r_i^{q\vee2}
\bigr)^2 \bigr] \esp^{1/2} \bigl[ {r'_l}^2
\bigr] .
\end{eqnarray*}
Applying~(\ref{eqappmax-qgeq2}), we obtain, for any $h\geq2$,
\begin{eqnarray*}
\esp \bigl[r_l^h \bigr] & =& \mathrm{O} \bigl(
\mu(C_{s,1-t})+\mu \bigl(D'_{0,s} \bigr) \bigr) =
\mathrm{O}(s+t) ,
\\
\esp \bigl[{r'_l}^h \bigr] & =& \mathrm{O} \bigl(
\mu(C_{s,1-t})+\mu \bigl(E'_{t,1} \bigr) \bigr) =
\mathrm{O}(s+t) .
\end{eqnarray*}
Thus finally
\[
\esp \bigl[\zeta_l^q\xi_l^{q'}
R_l \gamma'_l \bigr] \leq C(s+t) \esp
\bigl[\zeta_l^q \bigr] \esp \bigl[\xi_l^{q'}
\bigr] .
\]
The remaining terms in~(\ref{eqappcov-1}) and~(\ref{eqappcov-2}) are
dealt with similarly and we obtain
\[
\bigl|\cov \bigl(M^q \bigl([0,s] \bigr),M^{q'} \bigl([0,t]
\bigr) \bigr)\bigr| \leq C (s+t) \esp \bigl[\zeta_l^q \bigr]
\esp \bigl[\xi_l^{q'} \bigr] .
\]
The previous considerations also yield that
\begin{eqnarray*}
s^{\zeta(q)} &=& \esp \bigl[M^q \bigl([0,s] \bigr) \bigr]  =
\mathrm{e}^{\psi(q)} \esp \bigl[\zeta_l^q \bigr]
\bigl\{1+\mathrm{O}(s+t) \bigr\} ,
\\
t^{\zeta(q)} &=& \esp \bigl[M^{q'} \bigl([1-t,1] \bigr) \bigr]  =
\mathrm{e}^{\psi(q')} \esp \bigl[\xi_l^{q'} \bigr]
\bigl\{1+\mathrm{O}(s+t) \bigr\}
\end{eqnarray*}
and all the previous bounds finally yield~(\ref{eqappcov-plus-fort}).
\end{pf}

\begin{lem}
\label{lemappcov-differences}
If $2q < q_{\max}$, then for $k=1,\ldots,2^{n}-1$,
%
\begin{equation}
\label{eqappcov-differences} 2^{n\zeta(2q)} \esp[D_{0,0,n,q}
D_{0,k,n,q}] = \mathrm{O} \bigl(k^{-\{\psi(2q)-2\psi
(q)+1\}} \bigr) .
\end{equation}
\end{lem}

\begin{pf}
By the scaling property, and since $\esp[D_{0,k,n,q}]=0$, we have
\begin{eqnarray*}
&&2^{n\zeta(2q)} \esp[D_{0,0,n,q} D_{0,k,n,q}] \\
&&\quad = k^{\zeta(2q)}
\operatorname{cov} \biggl( M^q \biggl( \biggl[0,\frac1k \biggr]
\biggr),M^q \biggl( \biggl[1-\frac1{k},1 \biggr] \biggr) \biggr)
\\
&&\qquad{} - 2^{\tau(q)} \biggl(k-\frac12 \biggr)^{\zeta(2q)} \operatorname{cov}
\biggl( M^q \biggl( \biggl[0,\frac1{k-1/2} \biggr] \biggr),
M^q \biggl( \biggl[1-\frac{1}{2k-1},1 \biggr] \biggr) \biggr)
\\
&&\qquad{} - 2^{\tau(q)} k^{\zeta(2q)} \operatorname{cov} \biggl(
M^q \biggl( \biggl[0,\frac1k \biggr] \biggr), M^q \biggl(
\biggl[1-\frac1{2k},1 \biggr] \biggr) \biggr)
\\
&&\qquad{} - 2^{\tau(q)} k^{\zeta(2q)} \operatorname{cov} \biggl(
M^q \biggl( \biggl[0,\frac1{2k} \biggr] \biggr), M^q
\biggl( \biggl[1-\frac1{k},1 \biggr] \biggr) \biggr)
\\
&&\qquad{} + 2^{2\tau(q)} (k-1/2)^{\zeta(2q)} \operatorname{cov} \biggl(
M^q \biggl( \biggl[0,\frac1{2k-1} \biggr] \biggr), M^q
\biggl( \biggl[1-\frac1{2k-1},1 \biggr] \biggr) \biggr)
\\
&&\qquad{} + 2^{2\tau(q)} k^{\zeta(2q)} \operatorname{cov} \biggl(
M^q \biggl( \biggl[0,\frac1{2k} \biggr] \biggr), M^q
\biggl( \biggl[1-\frac1{2k},1 \biggr] \biggr) \biggr)
\\
&&\qquad{} - 2^{\tau(q)} k^{\zeta(2q)} \operatorname{cov} \biggl(
M^q \biggl( \biggl[\frac1{2k},\frac1k \biggr] \biggr),
M^q \biggl( \biggl[1-\frac1k,1 \biggr] \biggr) \biggr)
\\
&&\qquad{} + 2^{2\tau(q)} (k-1/2)^{\zeta(2q)} \operatorname{cov} \biggl(
M^q \biggl( \biggl[\frac1{2k-1},\frac2{2k-1} \biggr] \biggr),
M^q \biggl( \biggl[1-\frac1{2k-1},1 \biggr] \biggr) \biggr)
\\
&&\qquad{} + 2^{2\tau(q)} k^{\zeta(2q)} \operatorname{cov} \biggl(
M^q \biggl( \biggl[\frac1{2k},\frac1k \biggr] \biggr),
M^q \biggl( \biggl[1-\frac1{2k},1 \biggr] \biggr) \biggr) .
\end{eqnarray*}
Applying Lemma~\ref{lemappeven-better}, with $s$ and $t$ replaced by
$k$ and
$2k$ and $q=q'$, we obtain that each covariance term that appears
above is of order $k^{-2\zeta(q)-1}$, which yields $ 2^{n\zeta(2q)}
\esp[D_{0,0,n,q} D_{0,k,n,q}] = \mathrm{O}(k^{\zeta(2q)-2\zeta(q)-1})$, and since
$\zeta(2q) - 2 \zeta(q) = 2\psi(q)-\psi(2q)$, the
bound~(\ref{eqappcov-differences}) is proved.
\end{pf}

\begin{lem}
\label{lemappmoment-2p-Dnq}
If $4q<q_\chi$, then
\[
\esp \bigl[D_{0,n,q}^{4} \bigr] = \mathrm{O} \bigl(n2^{-n\tau(4q)} +
2^{-2n\tau(2q)} \bigr) .
\]
\end{lem}

\begin{pf}
Let us compute the fourth moment of $D_{0,n,q}$. For brevity, let the
centered random variables
$D_{0,k,n,q}$ be simply denoted by $x_k$. We have
%
\begin{eqnarray}
\label{eqappmoment-2p-Dnq} \esp \bigl[D_{0,n,q}^4 \bigr]& =&
2^n \esp \bigl[x_{0}^4 \bigr]  + \sum
_{0 \leq i \ne j \leq
2^n-1} \esp \bigl[x_i^2x_j^2
\bigr] + \sum_{0 \leq i \ne j \leq2^n-1} \esp \bigl[x_i^3x_j
\bigr]
\nonumber
\\[-8pt]
\\[-8pt]
\nonumber
&&{} + \mathop{\sum_{1 \leq i,j,k \leq2^n}}_{\#\{i,j,k\}=3} \esp \bigl[x_i^2x_jx_k
\bigr] + \mathop{\sum_{1 \leq i,j,k,l \leq2^n }}_{\#\{i,j,k,l\}=4} \esp [x_ix_jx_kx_l]
.
\end{eqnarray}
By the scaling property and Lemma~\ref{lemappeven-better}, obtain that
\[
2^{n\zeta(4q)} k^{-\zeta(4q)} \esp \bigl[x_1^2x_k^2
\bigr] = \mathrm{O} \bigl( k^{-2\zeta
(2q)} \bigr) .
\]
Since $\zeta(4q)<2\zeta(2q)$, this yields
\begin{eqnarray*}
\sum_{0 \leq i \ne j \leq2^n-1} \esp \bigl[x_i^2x_j^2
\bigr] = \mathrm{O} \Biggl( 2^{-n\tau(4q)} \sum_{k=0}^{2^n-1}
k^{\zeta(4q)-2\zeta(2q)} \Biggr) = \mathrm{O} \bigl( n 2^{-n\tau
(4q)}+2^{-2n\tau(2q)} \bigr)
.
\end{eqnarray*}
Again, by Lemma~\ref{lemappeven-better}, we have
\begin{eqnarray*}
2^{n\zeta(4q)} k^{-\zeta(4q)} \esp \bigl[x_1^3x_k
\bigr] = 2^{n\zeta(4q)} k^{-\zeta(4q)} \cov \bigl(x_1^3,x_k
\bigr) = \mathrm{O} \bigl( k^{-\zeta(3q)-\zeta(q)-1} \bigr) .
\end{eqnarray*}
By~(\ref{eqpsi-pq}), if $4q<q_\chi$, then $\psi(4q)>4\psi(3q)/3$ and
$\psi(3q)/3>\psi(q)$, so $\zeta(4q)-\zeta(3q)-\zeta(q)<0$, thus
\begin{eqnarray*}
\sum_{0 \leq i \ne j \leq2^n-1} \esp \bigl[x_i^3x_j
\bigr] = \mathrm{O} \Biggl( 2^{-n\tau(4q)} \sum_{k=0}^{2^n-1}
k^{\zeta(4q)-\zeta(3q)-\zeta(q)-1} \Biggr) = \mathrm{O} \bigl(2^{-n\tau(4q)} \bigr) .
\end{eqnarray*}
We now calculate the fourth term in the expansion~(\ref{eqappmoment-2p-Dnq})
of $\esp[D_{0,n,q}^4]$. By stationarity we may assume $i=0$ and without
loss of
generality assume $j<k/2$. Set $y_\ell=D_{0,\ell,\log_2(k),q} $ for
$\ell=1,\ldots, k$. Then by the scaling property
\[
\esp \bigl[x_i^2x_jx_k
\bigr] = \bigl(k/2^n \bigr)^{\zeta(4q)} \esp \bigl[y_1^2y_jy_k
\bigr] .
\]
Since $\esp[y_k] = 0$, from the definition of $D_{\ell,\log_2(k),q}$ we may
write
%
\begin{eqnarray}
\label{eqappproof-fourthterm-1} \esp \bigl[y_1^2y_jy_k
\bigr] &=& \cov \bigl(y_1^2y_j,y_k
\bigr)
\nonumber\\
&=& \sum_{l,s,t} \beta_l
\alpha_s \eta_t
\cov \bigl(M^{r_l q}(\Delta_{1, \log_2(b_l k)}) M^{(2-r_l) q}(
\Delta_{1, \log_2(b_l
k)})\\
&&\hspace*{4pt}\quad{}\times M^{ q}(\Delta_{j, \log_2(b_s k)}),
M^{ q}(\Delta_{k, \log_2(b_t
k)}) \bigr) ,\nonumber
\end{eqnarray}
where $r_l\in\{1,2\}$ and $b_l,b_s,b_t\in\{1,2\}$ indicate whether the
scale is
$k$ or $2k$. Set $\ell=1-j/k$. In the notation of
Lemmas~\ref{lemappapprox-MnDeltan} and~\ref{lemappeven-better} set
$C=A_\ell((j-1)/k,j/k)\cap A_\ell((k-1)/k,1)$, $A_{1}=A_\ell(1/k,(j-1)/k)$,
$A_{2}=A_\ell((j-1)/k, j/k)\cap A_\ell(0,1/k)$ and
$A_{3}=B_\ell(j/k,(k-1)/k)$. So that $A_i\cap A_3=\varnothing$ for
$i=1,2$ and
$A_i\cap C=\varnothing$ for $i=1,2,3$. Also define
\begin{eqnarray*}
\zeta_{l,1} & = M^{r_l q}_\ell(
\Delta_{i, \log_2(b_l k)}) M^{q}_\ell (\Delta_{j,\log_2(b_s k)})
, \qquad \zeta_{l,2} = M^{ q}_\ell(\Delta_{k,\log
_2(b_t k)})
,
\end{eqnarray*}
which by construction are independent of $\mathrm{e}^{P(A_i)}$ and
$\mathrm{e}^{P(C)}$. Then
\begin{eqnarray*}
M^{r_l q}(\Delta_{i,\log_2(b_l k)}) M^{ q}(
\Delta_{j,\log_2(b_s k)}) & =& \mathrm{e}^{qr_l(P(A_1)+P(A_2))} \mathrm{e}^{q(P(A_1)+P(C))}
\zeta_{l,1} \times \{ 1 + q\gamma_{l,1} + R_{l,1} \}
,
\\
M^{ q}(\Delta_{k,\log_2(b_t k)}) & =& \mathrm{e}^{q(P(A_3)+P(C))}
\zeta_{l,2} \times \{ 1 + q\gamma_{l,2}+R_{l,2} \} ,
\end{eqnarray*}
where $\gamma_{l,i}$ and $R_{l,i}$ are independent of $\zeta_{l,i}$,
$\mathrm{e}^{P(A_i)}$ and $\mathrm{e}^{P(C)} $ and satisfy
$\esp[\gamma_{l,1}\gamma_{l,2}]=\mathrm{O}(1/k)$, $\esp[\gamma_{l,i}R_{l,i}]=\mathrm{O}(1/k)$ and
$\esp[R_{l,1}R_{l,2}]=\mathrm{O}(1/k)$. Finally set
$K_{l,1}=\break \esp[\zeta_{l,1}\mathrm{e}^{qr_l(P(A_1)+P(A_2))}\times  \mathrm{e}^{qP(A_1)}]$
and $K_{l,2}=\esp[\zeta_{l,2} \mathrm{e}^{qP(A_3)}]$. Then, for each of the
terms in (\ref{eqappproof-fourthterm-1})
%
\begin{eqnarray}
\label{eqappproof-fourthterm-2}
&&\cov \bigl(M^{r_l q}(
\Delta_{i,\log_2(b_l k)})  M^{q}( \Delta_{j,\log_2(b_s
k)}),
M^{q}(\Delta_{k,\log_2(b_t k)}) \bigr)
\nonumber
\\
&&\quad= K_{l,1} K_{l,2} \var(C) \bigl(1+\mathrm{O}(1/k) \bigr)
\nonumber
\\[-8pt]
\\[-8pt]
\nonumber
&&\quad = \frac{1}{\esp^2[\mathrm{e}^{qP(C)}]} \esp \bigl[ M^{r_l q}(\Delta_{i,\log_2(b_l k)})
M^{ q}(\Delta_{j,\log_2(b_s k)}) \bigr]
\\
&&\qquad{} \times\esp \bigl[M^{ q}_\ell(
\Delta_{k,\log_2(b_t k)}) \bigr]\log (1-j/k) \bigl(1+\mathrm{O}(1/k) \bigr)^2 .\nonumber
\end{eqnarray}
Adding up all the terms in (\ref{eqappproof-fourthterm-1}) and using
$\esp
[M^{q}_\ell(\Delta_{k,\log_2(b_t k)})] = \mathrm{O}(k^{-\zeta(q)})$ for all $\ell$
yields
\[
\esp \bigl[y_1^2y_jy_k
\bigr] = \mathrm{O}(j/k) \esp \bigl[y_1^2y_j
\bigr]k^{\zeta(q)} .
\]
On the other hand, again by the scaling property, and because $\esp[y_j]=0$,
$\esp[y_1^2y_j]= \cov(y_1^2,y_j)$ and applying
Lemma~\ref{lemappeven-better} we have
%
\begin{equation}
\label{eqappproof-fourthterm-3} \esp \bigl[y_1^2y_j
\bigr] = \mathrm{O} \bigl(k^{-\zeta(3q)}j^{\zeta(3q)-\zeta(2q)-\zeta(q)-1} \bigr) .
\end{equation}
By (\ref{eqappproof-fourthterm-2}) and (\ref{eqappproof-fourthterm-3})
we obtain the bound
\[
\esp \bigl[x_0^2x_jx_k
\bigr]  = \mathrm{O} \bigl(2^{-n\zeta(4q)} j^{\zeta(3q)-\zeta(2q)-\zeta(q)} k^{\zeta(4q)-\zeta(3q)-\zeta(q)-1} \bigr) .
\]
Noting that by convexity of $\psi$, it holds that $2\psi(q)<\psi(2q)$,
this yields
\[
\mathop{\sum_{1 \leq i,j,k \leq2^n }}_{\#\{i,j,k\}=3} \esp \bigl[x_i^2x_jx_k
\bigr]  = \mathrm{O} \bigl(2^{-2n\tau(2q)} + n 2^{-n\tau(4q)} \bigr) .
\]
For the last term in (\ref{eqappmoment-2p-Dnq}) by stationarity set
$i=0$, and
assume $j<\ell<k$ and moreover that $\ell-j<k/2$. Write
\[
\esp[x_ix_jx_\ell x_k] =
\cov(x_i x_j, x_\ell x_k) +
\esp[y_iy_j] \esp [y_\ell
y_k] .
\]
The term $\cov(y_1y_j,y_\ell y_k)$ can be shown to be of smaller order
than the
product of expectations. Thus, applying Lemma~\ref{lemappcov-differences},
we finally obtain
\[
\hspace*{105pt}\mathop{\sum_{1 \leq i,j,k,l \leq2^n }}_{\#\{i,j,k,l\}=4} \esp [x_ix_jx_kx_\ell]
= \mathrm{O} \bigl(2^{-2n\tau(q)} \bigr) .\hspace*{105pt}\qed
\]
\noqed\end{pf}

\textit{Bounds for the MRW, case $H>1/2$}.
Define $\tilde{a}_{j,k,n,H} = \mathrm{e}^{w_{l_n}(t_{j,k})} \tilde
\delta_{j,k,n,H}$ with
\[
\tilde\delta_{j,k,n,H}^2 = \int_{\Delta_{k,n}^{(j)}} \int
_{\Delta_{k,n}^{(j)}} |u-v|^{2H-2} M_n(\mathrm{d} u)
M_n(\mathrm{d} v)
\]
and for $j_1\ne j_2$,
\[
\tilde{\rho}_{H} \bigl(j_1,j_2,k,k'
\bigr) = \frac{\int_{\Delta_{k,n}^{(j_1)}}
\int_{\Delta_{k',n}^{(j_2)}} |u-v|^{2H-2} M_n(\mathrm{d} u)
M_n(\mathrm{d}v) }{ \delta_{j_1,k,n,H} \delta_{j_2,k,n,H} } .
\]
%

\begin{lem}
\label{lemappcotadef1}
For $p\ge1$ such that $2pq<q_\chi$ and for $r\geq2$, there exist $\eta,C>0$
and uniformly bounded constants $c_{q,H}(k,k')$ such that
%
\begin{eqnarray}
\bigl\llvert 2^{n\zeta_H(2q)} \mathrm{e}^{\psi(2q)} l_n^{-\psi(2q)}
\esp \bigl[ \tilde\delta_{j,k,n,H}^{2q} \bigr] -
m_H(2q) \bigr\rrvert &=& \mathrm{O} \bigl(2^{-n\eta} \bigr) ,
\label{eqappmndeltan-H}
\\
\bigl\llvert 2^{n\zeta_H(2q)} \mathrm{e}^{\psi(2q)} l_n^{-\psi(2q)}
\esp \bigl[ \tilde \delta_{0,k,n,H}^q \tilde
\delta_{0,k',n,H}^q \bigr] - c_{q,H}
\bigl(k,k' \bigr) \bigr\rrvert &=& \mathrm{O} \bigl(2^{-n\eta} \bigr) ,
\label{eqappcota1492}
\\
2^{n\zeta_H(2pq)} \esp \bigl[ \bigl|a_{0,k,n,H}^{2q} -
\tilde{a}_{0,k,n,H}^{2q}\bigr|^p \bigr] &=&\mathrm{O}
\bigl(2^{-n\eta} \bigr) , \label{eqappcota164}
\\
2^{n\zeta_H(2pq)} \esp \bigl[\bigl|a_{0,k,n,H}^q
a_{0,k',n,H}^q - \tilde{a}_{0,k,n,H}^q
\tilde{a}_{0,0,k',n,H}^q\bigr|^p \bigr]& =&\mathrm{O}
\bigl(2^{-n\eta} \bigr) . \label{eqappcota162}
\end{eqnarray}

\end{lem}
\begin{pf}
Note that~(\ref{eqappcota162}) implies~(\ref{eqappcota1492})
and~(\ref{eqappcota164}) implies~(\ref{eqappmndeltan-H}). By
stationarity of
increments, we can assume without loss of generality that $k'=0$. For
brevity, denote $a_k = a_{0,k,n,H}$, $\tilde a_k = \tilde a_{0,k,n,H}$ and
$\tilde\delta_k = \tilde\delta_{0,k,n,H}$. Generalizing the notation
of the
proof of Lemma~\ref{lemappapprox-MnDeltan}, we can write $ a_k^2 =
\xi_k^2(R_k+1)$ with $\xi_k = \mathrm{e}^{P(I_n(k))} \tilde\delta_k$, $I_n(k)
= A_{l_n}(k2^{-n}) \setminus A_{l_n}(2^{-n})$, $B_k(u) =
A_{l_n}(u)\setminus
I_n(k)$ and
\begin{eqnarray*}
R_k = \int_{\Delta_{k,n}} \int_{\Delta_{k,n}}
\bigl\{\mathrm e^{P(B_k(u))+P(B_k(v))} - 1 \bigr\} |u-v|^{2H-2} \tilde
M_k(\mathrm{d} u) \tilde M_k(\mathrm{d} v) .
\end{eqnarray*}
Denote $r_k = \sup_{u\in\Delta_{k,n}} |\mathrm{e}^{P(B_k(u))}-1|$. Then $|R_k|
\leq(1+r_k)^2-1$, the sequence $\{r_k,k=0, \ldots,2^{n}-1\}$ is
independent of
the measures $\tilde M_k$, $0 \leq k \leq2^n-1$ and by~(\ref{eqappmax-qgeq2})
and H\"older's inequality, we have, for $p\geq1$, $\esp[|r_0|^p] =
\mathrm{O}(\sqrt{\mu(B_0(2^{-n})}) = \mathrm{O}(2^{-\alpha n/2})$. Thus
\begin{eqnarray*}
\esp \bigl[ \bigl\llvert a_k^{q} -
\mathrm{e}^{qP(I_n(k))} \tilde\delta_k^q \bigr
\rrvert^p \bigr] \leq\esp \bigl[\mathrm{e}^{pqP(I_n(k))} \bigr] \esp
\bigl[\tilde\delta_0^{pq} \bigr] \mathrm{O} \bigl(2^{-\alpha n}
\bigr) ,
\end{eqnarray*}
which proves~(\ref{eqappcota164}). Since $\esp[\mathrm{e}^{qP(I_n(k))}]
\sim
\esp[\mathrm{e}^{qw_{l_n}(0)}] = \mathrm{e}^{\psi(q)} l_n^{-\psi(q)}$,
this implies
that $\esp[\tilde\delta_0^q] \sim cl_n^{\psi(q)} 2^{-n\zeta_H(q)}$. Next,
using the bound $|(1+x)^q-1| \leq C (|x| + |x|^{q\wedge1})$ valid for
$x\geq0$, we obtain
\begin{eqnarray*}
\esp \bigl[ \bigl|a_0^qa_k^q-
\xi_0^q \xi_k^q\bigr|^p
\bigr] & \leq&\esp \bigl[ \xi_0^{pq}\xi_k^{pq}
\bigl|(R_0+1)^{q/2}(R_k+1)^{q/2}-1\bigr|^p
\bigr]
\\
& \leq&\esp \bigl[ \xi_0^{pq}\xi_k^{pq}
\bigr] \esp \bigl[\bigl|(r_0+1)^{q}(r_k+1)^{q}-1\bigr|^p
\bigr] \leq C 2^{-\eta n} \esp \bigl[ \xi_0^{2pq}
\bigr]
\end{eqnarray*}
for some $\eta>0$. This proves~(\ref{eqappcota162}).
\end{pf}

\begin{lem}
\label{lemappcov-diff-H>12} If $2q<q_\chi$, then
%
\begin{equation}
\label{eqappcov-diff-H>12} 2^{n\zeta_H(2q)} \bigl|\esp[U_{0,n,0},U_{0,n,k}]
\bigr| \leq C k^{-\{\psi
(2q)-2\psi(q)+1\}} .
\end{equation}

\end{lem}

\begin{pf}
For $k\geq1$, denote
\begin{eqnarray*}
U_{k} & = &\int_0^{1/k} \int
_0^{1/k} |u - v|^{2H-2} M(\mathrm{d} u)
M(\mathrm{d} v) ,\\
 U'_{k} &= &\int_{1/2k}^{1/k}
\int_{1/2k}^{1/k} |u-v|^{2H-2} M(\mathrm{d}
u) M(\mathrm{d} v) ,
\\
V_{k} & =& \int_{1-1/k}^1 |u -
v|^{2H-2} M(\mathrm{d} u) M(\mathrm{d} v) ,\\
 V'_{k}
&= &\int_{1-1/k}^{1-1/2k} \int_{1-1/k}^{1-1/2k}
|u-v|^{2H-2} M(\mathrm{d} u) M(\mathrm{d} v) .
\end{eqnarray*}
Then, by the scaling property, we have
\begin{eqnarray*}
&&2^{n\zeta_H(2q)} \esp[U_{0,n,0}U_{0,n,k}]\\
 &&\quad = k^{\zeta_H(2q)}
\cov \bigl(U_{k}^q,V_{k}^q \bigr)
- 2^{\tau_H(q)} (k-1/2)^{\zeta_H(2q)} \cov \bigl(U_{k}^q,V_{2k}^q
\bigr)
\\
&&\qquad{} - 2^{\tau_H(q)} k^{\zeta_H(2q)} \bigl\{ \cov \bigl(U_{k}^q,{V'_{k}}^q
\bigr) - \cov \bigl(U_{2k}^q,V_{k}^q
\bigr) + \cov \bigl({U'_{k}}^q,V_{k}^q
\bigr) \bigr\}
\\
&&\qquad{} + 2^{2\tau_H(q)} (k-1/2)^{\zeta_H(2q)} \bigl\{\cov \bigl(U_{2k}^q,V_{2k}^q
\bigr)+\cov \bigl({U'_{k}}^q,V_{2k}^q
\bigr) \bigr\}
\\
&&\qquad{} + 2^{2\tau_H(q)} k^{\zeta_H(2q)} \bigl\{\cov \bigl(U_{2k}^q,{V'_{k}}^q
\bigr)+\cov \bigl({U'_{k}}^q,{V'_{k}}^q
\bigr) \bigr\} .
\end{eqnarray*}
All the covariance terms are of the same order, and we only consider
the first
one, $\cov(U_k^q,V_k^q)$. Denote $l=1-2/k$, define the measure
$M_l(\mathrm{d} u) = \mathrm
e^{-w_l(u)}M(\mathrm{d} u)$ and
\begin{eqnarray*}
\zeta_{k,H} & = &\int_0^{1/k} \int
_0^{1/k} |u-v|^{2H-2} M_l(
\mathrm d u) M_l(\mathrm{d} v) ,
\\
\xi_{k,H} & =& \int_{1-1/k}^1 \int
_{1-1/k}^1 |u-v|^{2H-2} M_l(
\mathrm d u) M_l(\mathrm{d} v) ,
\\
A_k & =& A_l(1/k) \setminus A_l(1-1/k),\qquad
B_k = A_l(1-1/k) \setminus A_l(1/k) ,
\\
\bar A_k(u) & =& A_l(u) \setminus A_k,\qquad
\bar B_k(u) = A_l(u)\setminus B_k ,
\\
\pi_k(u,v) & =& P_0 \bigl(\bar A_k(u)
\bigr)+ P_0 \bigl(\bar A_k(v) \bigr) ,\\
\pi'_k(u,v) &= &P_0 \bigl(\bar
B_k(u) \bigr)+ P_0 \bigl(\bar B_k(v)
\bigr) ,
\\
\tilde\alpha_k & =& \zeta_{k,H}^{-1} \int
_0^{1/k} \int_0^{1/k}
|u-v|^{2H-2} \pi_k(u,v) M_l(\mathrm{d} u)
M_l(\mathrm{d} v) ,
\\
\tilde\beta_k & =& \xi_{k,H}^{-1} \int
_{1-1/k}^1 \int_{1-1/k}^1
|u-v|^{2H-2} \pi'_k(u,v) M_l(
\mathrm{d} u) M_l(\mathrm{d} v) .
\end{eqnarray*}
Then we can write
\begin{eqnarray*}
U_k^q & =& \mathrm{e}^{2qP(A_k)}
\zeta_{k,H}^q + \mathrm{e}^{2qP(A_k)}
\zeta_{k,H}^q \tilde\alpha_k +
\mathrm{e}^{2qP(A_k)} \zeta_{k,H}^q R_k ,
\\
V_k^q & = &\mathrm{e}^{2qP(B_k)}
\xi_{k,H}^q + \mathrm{e}^{2qP(B_k)}
\xi_{k,H}^q \tilde\beta_k +
\mathrm{e}^{2qP(B_k)} \xi_{k,H}^q R'_k
,
\end{eqnarray*}
with
\begin{eqnarray*}
R_k & =& \biggl( \int_0^{1/k}
\int_0^{1/k} |u-v|^{2H-2} \mathrm
{e}^{P(\bar A_k(u))+P(\bar A_k(v))} \tilde M_l(\mathrm{d} u) \tilde
M_l(\mathrm{d} v) \biggr)^q - 1 -q \tilde
\alpha_k ,
\\
R'_k & =& \biggl( \int_{1-1/k}^1
\int_{1-1/k}^1 |u-v|^{2H-2} \mathrm
e^{P(\bar B_k(u))+P(\bar B_k(v))} \tilde M'_l(\mathrm{d} u) \tilde
M'_l(\mathrm{d} v) \biggr)^q - 1 - q
\tilde\beta_k .
\end{eqnarray*}
Note that $P(\bar A_k(u))$, $P(\bar B_k(u))$, $\zeta_k$ and $\xi_k$ are mutually
independent and by~(\ref{eqapptrivial}),
\begin{eqnarray*}
|R_k| & \leq &C \sup_{u,v\in[0,1/k]} \bigl\llvert
\mathrm{e}^{P(\bar
A_k(u))+P(\bar
A_k(v))}-1 \bigr\rrvert^{q\vee2} + C \sup_{u,v\in[0,1/k]}
\bigl\llvert \mathrm {e}^{P(\bar A_k(u)) + P(\bar A_k(v))}-1 \bigr\rrvert^{2}
\\
&&{} + C \sup_{u,v\in[0,1/k]} \bigl\llvert \mathrm{e}^{P(\bar A_k(u))+P(\bar
A_k(v))}-1-
\pi_k(u,v) \bigr\rrvert .
\end{eqnarray*}
Applying now the bounds~(\ref{eqappmax-qgeq2}) and~(\ref{eqappideal-bound})
we obtain that
\begin{eqnarray*}
\esp \bigl[U_{k}^q \bigr] & = &\mathrm{e}^{2\psi(q)}
\esp \bigl[ \zeta_{k,H}^q \bigr] \bigl\{1+\mathrm{O}
\bigl(k^{-1} \bigr) \bigr\} + q \mathrm{e}^{2\psi(q)} \esp \bigl[
\zeta_{k,H}^q \tilde \alpha_k \bigr] ,
\\
\esp \bigl[V_{k}^q \bigr] & = &\mathrm{e}^{2\psi(q)}
\esp \bigl[ \xi_{k,H}^q \bigr] \bigl\{ 1+\mathrm{O}
\bigl(k^{-1} \bigr) \bigr\} + q \mathrm{e}^{2\psi(q)} \esp \bigl[
\xi_{k,H}^q \tilde\beta_k \bigr] ,
\\
\esp \bigl[U_{k}^q V_{k}^q
\bigr] & =& \mathrm{e}^{4\psi(q)} \esp \bigl[ \zeta_{k,H}^q
\bigr] \esp \bigl[ \xi_{k,H}^q \bigr] \bigl\{1+\mathrm{O}
\bigl(k^{-1} \bigr) \bigr\} + q \mathrm{e}^{2\psi(q)} \bigl\{ \esp
\bigl[ \xi_{k,H}^q \bigr] \esp \bigl[\zeta_{k,H}^q
\tilde\alpha_k \bigr]\\
&&{} + \esp \bigl[ \zeta_{k,h}^q
\bigr] \esp \bigl[\xi_{k,H}^q \tilde\beta_k
\bigr] \bigr\} .
\end{eqnarray*}
Combining these bounds yields the requested bound for $\cov(U_k^q,V_k^q)$
and~(\ref{eqappcov-diff-H>12}).
\end{pf}
\end{appendix}

%


\printhistory

\end{document}